\documentclass[11pt]{article}

\usepackage{graphicx} 
\usepackage[hidelinks]{hyperref}
	\usepackage{amssymb}
\usepackage{caption}
\usepackage{subcaption}
	\usepackage{amsmath}
	\usepackage{amsfonts}
	\usepackage{amsthm}
	\usepackage{stmaryrd}
        \usepackage{mathtools}
	\usepackage{algorithm, algorithmic,refcount}
	\usepackage{tkz-euclide}

	\usepackage{epstopdf}
	\usepackage{mathdots}
	\usepackage{appendix}
 \usepackage[shortlabels]{enumitem}
 \usetikzlibrary{patterns}
\usepackage{xcolor}
 \usepackage{natbib}
	\setcitestyle{square}
	\setcitestyle{numbers}
	\setcitestyle{comma}
 \usepackage{graphicx}
 \usepackage{cleveref}
 \usepackage{booktabs}

 \newtheorem{theorem}{Theorem}
	\newtheorem{example}[theorem]{Example}
	\newtheorem{corollary}[theorem]{Corollary}
	\newtheorem{lemma}[theorem]{Lemma}
	\newtheorem{remark}[theorem]{Remark}

\newcommand{\ie}{i.e., }

\newcommand{\calO}{\mathcal{O}}

\newcommand{\E}{\mathbb{E}}
\newcommand{\var}{\mathrm{Var}}

\DeclareMathOperator{\tr}{\mathrm{tr}}
\DeclareMathOperator{\diag}{\mathrm{diag}}

\newcommand{\KV}{\hat \theta_{2p}}
\newcommand{\X}{\hat \Theta_{2p}}
\newcommand{\Y}{Y_{i_1,\ldots,j_p}^{t_1,\ldots,s_p}}

\newcommand{\rev}[1]{\textcolor{black}{#1}}

\title{Improved bounds for randomized Schatten norm estimation of numerically low-rank matrices}

\author{Ya-Chi Chu\footnote{Corresponding author} \footnote{Department of Mathematics, Stanford University, CA, USA. Email: \texttt{ycchu97@stanford.edu}} \and Alice Cortinovis\footnote{Department of Mathematics, Stanford University, CA, USA. Email: \texttt{alicecrvs@gmail.com}}}

\date{}

\begin{document}

\maketitle

\begin{abstract}
In this work, we analyze the variance of a stochastic estimator for computing Schatten norms of matrices. The estimator extracts information from a single sketch of the matrix, that is, the product of the matrix with a few standard Gaussian random vectors. While this estimator has been proposed and used in the literature before, the existing variance bounds are often pessimistic. 
\rev{Our work provides a new upper bound and estimates of the variance of this estimator. These theoretical findings are supported by numerical experiments, demonstrating that the new bounds are significantly tighter than the existing ones in the case of numerically low-rank matrices.}
\end{abstract}



\section{Introduction}

Given a matrix $A \in \mathbb{R}^{m \times n}$, we consider the problem of numerically approximating its Schatten-$2p$ norm
\begin{equation*}
    \|A\|_{2p} = \left ( \sigma_1^{2p} + \sigma_2^{2p} + \cdots + \sigma_{\min\{n,m\}}^{2p} \right )^{\frac{1}{2p}},
\end{equation*}
where $\sigma_1, \sigma_2, \ldots, \sigma_{\min\{n,m\}}$ are the singular values of $A$ and $p$ is a positive integer. While the Schatten-$2p$ norm can be readily computed from the singular values of the matrix, computing the singular values themselves generally requires $\mathcal O(nm \min\{n,m\})$ time and is infeasible for large-scale matrices. 

Several randomized methods have been developed to compute an \emph{approximation} of the Schatten norm with cheaper operations, trading some accuracy for speed. A class of such methods is based on the Hutchinson trace estimator~\cite{Hutchinson1989}, a randomized algorithm that approximates the trace of a square matrix $B \in \mathbb{R}^{n \times n}$ by averaging a few quadratic forms $\omega_i^{T} B \omega_i$, where $\omega_i$'s are isotropic random vectors of length $n$. The Hutchinson estimator is often used when $B$ is a matrix function~\cite{Cortinovis2022,Ubaru2017} and relates to the Schatten-$2p$ norm estimation problem via the formula
\begin{equation}\label{eq:hutch}
    \|A\|_{2p}^{2p} = \mathrm{trace}\left ((A^TA)^p\right ).
\end{equation}
A variety of variance reduction techniques has been developed, for instance the combination of Hutchinson trace estimator with randomized low-rank approximation; see, e.g.,~\cite{Epperly2024,Meyer2021,Persson2022}. Algorithms specifically tailored to Schatten-$p$ norm estimation have been discussed in~\cite{Dudley2022,Han2017}.
When computing quadratic forms involving $(A^TA)^{p}$, \rev{at least $p$ matrix-vector products (matvecs) with $A$ are required and must be applied sequentially.}
This number can be slightly reduced if \rev{the quadratic forms themselves }are approximated with Chebyshev polynomials as in~\cite{Dudley2022,Han2017}.
\rev{Another line of works \cite{braverman2018matrix,braverman2020schatten} assumes that the rows of $A$ are available one at a time -- this is known as the \emph{row-order model}. Under this setting, a Schatten-$2p$ norm estimator can be computed after one or multiple passes over the matrix $A$.}

\rev{The \emph{sketching model} assumes access to a \emph{sketch} of the matrix $A$, that is, a matrix $Y := A\Omega$ for some random matrix $\Omega \in \mathbb{R}^{n \times k}$, but not necessarily assumes the access to the matvec with $A$ and to the entries of $A$.
This setting applies, for instance, to a problem considered in~\cite{kong2017spectrum}: 
approximating the set of the eigenvalues $\sigma_1^2, \ldots, \sigma_n^2$ of an unknown covariance matrix $\Sigma = A^T A$ given access to $k \ll n$ samples drawn from a multivariate distribution corresponding to $\Sigma$. This problem can be reduced, via the methods of moments \cite{bai2010estimation,li2014local,rao2008statistical}, to estimate the $p$-th moment of the spectrum of the covariance matrix $\Sigma$, defined as $\sum_{i=1}^n \sigma_i^{2p} = \| A \|_{2p}^{2p}$.
In our work, we focus on the case where the $k$ samples are drawn from a \emph{multivariate Gaussian distribution}. Such samples correspond to a \emph{Gaussian sketch} $Y := A \Omega$ where $\Omega \in \mathbb{R}^{n \times k}$ consists of i.i.d. standard Gaussian entries. In the language of randomized numerical linear algebra, the problem from~\cite{kong2017spectrum} can be reformulated as estimating the Schatten-$2p$ norms of a matrix $A$ given a Gaussian sketch of $A$. In case the matrix $A$ is known, but very large, the sketching model can also be advantageous since it allows for parallelization of the matrix-vector multiplications with $A$.}

\rev{The trace estimation techniques mentioned above are not applicable under the sketching model, so other strategies need to be employed. Li, Nguyen, and Woodruff~\cite{li2014sketching} showed that it is possible to extract an unbiased estimate of $\|A\|_{2p}^{2p}$ from a sketch of $A$. Kong and Valiant~\cite{kong2017spectrum} proposed a computationally efficient way to compute a similar unbiased estimator, which is the focus of our work.}
\rev{In Algorithm~\ref{alg:KV} we summarize the procedure for constructing the estimator for $\|A\|_{2p}^{2p}$ proposed in~\cite{kong2017spectrum}, using a Gaussian sketch. The estimator is denoted by $\KV$.}
In Algorithm~\ref{alg:KV}, line 2 uses the MATLAB function notation \texttt{triu($\cdot$,1)} to indicate the strictly upper triangular part of a matrix. \rev{Note that we assume $p \le k$, otherwise $T^{p-1}=0$ and the estimator would be identically zero.}
\begin{algorithm}
\caption{Estimator for Schatten-$2p$ norm from \cite{kong2017spectrum}}
\label{alg:KV}
\begin{algorithmic}[1] 
\REQUIRE{\rev{Gaussian s}ketch $Y = A \Omega \in \mathbb{C}^{m \times k}$ and \rev{positive} integer \rev{$p \le k$}}
\ENSURE{Unbiased estimator $\KV$ for $\|A\|_{2p}^{2p}$} 

\vspace{0.5pc}

\STATE Compute $Z = Y^T Y$.

\STATE Define $T = \texttt{triu}(Z,1)$ as the strictly upper triangular part of the matrix $Z$.

\STATE Compute $\KV = \binom{k}{p}^{-1} \mathrm{trace}(T^{p-1} Z)$
\end{algorithmic}
\end{algorithm}

It was proved in~\cite[Proposition 4]{kong2017spectrum} that the variance of $\KV$ \rev{generated by Algorithm~\ref{alg:KV}} 
satisfies the following bound:
\begin{equation}\label{eq:varKV-original}
    \mathrm{Var}(\KV) \le 2^{12p} p^{6p}3^p \max \left \{ \frac{n^{p-2}}{k^p}, \frac{1}{k}, \frac{n^{\frac{1}{2}-\frac{1}{p}}}{k} \right \} \|A\|_{2p}^{4p}.
\end{equation}
\rev{Lower bounds on Schatten norm estimation~\cite{li2014sketching} imply that the dependence on $n$ and $k$ featured in~\eqref{eq:varKV-original} cannot be improved in the worst case. However, the bound~\eqref{eq:varKV-original}} severely overestimates the true variance of $\KV$ \rev{for many matrices in practice}. As an illustrative example, Figure~\ref{fig:firstexample} plots the variance for $\hat{\theta}_{6}$ (estimator of the $6$-th power of the Schatten-$6$ norm) on a $10 \times 10$ matrix with singular values $1, 1/2, 1/3, \ldots, 1/10$, together with the bound~\eqref{eq:varKV-original}. \rev{Note that the bound~\eqref{eq:varKV-original}} overestimates the variance by several orders of magnitude. 

\begin{figure}[htb]
    \centering
    \includegraphics[scale=.3]{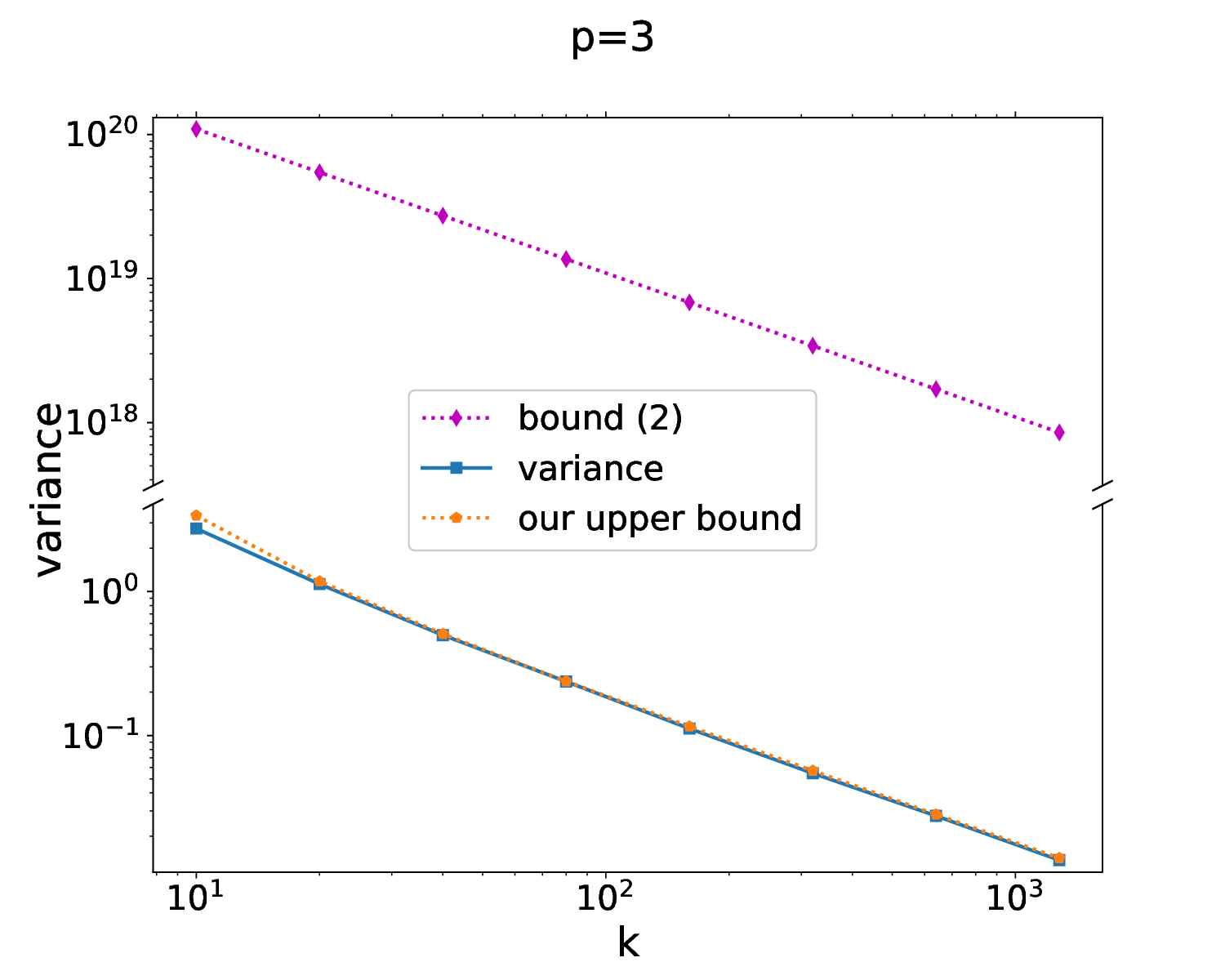}
    \caption{Comparison between \rev{the} true variance of $\hat{\theta}_{6}$\rev{,} the variance bound \eqref{eq:varKV-original}\rev{, and our bound from Theorem~\ref{thm:secondorderbound}}  for a $10 \times 10$ matrix $A$. \rev{Note that the $y$-axis has been ``cut'' for improved readability.}}
    \label{fig:firstexample}
\end{figure}

\paragraph{Contributions.} In this paper, we are interested in better understanding the variance of $\KV$ and its behavior on matrices with strongly decaying singular values (\ie numerically low-rank matrices). 
We derive \rev{a bound on the variance of $\KV$ that improves~\eqref{eq:varKV-original} by a large margin for all the matrices we have considered; see Figure~\ref{fig:firstexample} above for an illustrative example. While the large constants in the bound~\eqref{eq:varKV-original} were due, in part, to the fact that the analysis was done in a more general case, in which the entries of $\Omega$ could be any centered random variable with unit variance and bounded fourth moment, our bounds instead leverage the rotational invariance of the Gaussian distribution to simplify the analysis of the estimator. The results we obtain do not depend explicitly on the size of the matrix, and they only} depend on the singular values of $A$. 
Numerical results show that \rev{our} bound is relatively tight for matrices with rapid singular value decay using moderately large sketch size $k$. Moreover, in all \rev{the numerical experiments} our variance bound is \rev{much} more precise than the previous bound~\eqref{eq:varKV-original}. We provide a first-order and second-order expansion of the variance of $\KV$ in $1/k$. \rev{In particular, we} will show in \Cref{thm:firstorder} that, with a Gaussian sketching matrix $\Omega$, the variance of $\KV$ has first-order expansion
\begin{equation} \label{eq:varKV-ours}
\var(\KV) = \frac{2p^2}{k} \|A\|_{4p}^{4p} + \calO \left( \frac{1}{k^2} \right).
\end{equation} The experiments verify that our first-order and second-order estimates capture the behavior of $\KV$ for matrices that exhibit a strong singular value decay, for moderate values of $k$.

\paragraph{Outline.} Section~\ref{sec:KV} reviews the building blocks of $\KV$. Section~\ref{sec:estimates} presents our main results on the variance of $\KV$. In Section~\ref{sec:examples}, we illustrate the performance of our bounds and estimates on a variety of examples with or without singular value decay. Our conclusions are summarized in Section~\ref{sec:conclusion}.

\section{Notations and review of estimator $\KV$}~\label{sec:KV}
This short section defines the notations throughout the paper and introduces a mathematically equivalent expression of $\KV$, which will be useful in our analysis later.
Given a positive integer $k$, we denote $[k] = \{1, 2, \ldots, k\}$. 
Given another positive integer $p \leq k$, a \emph{$p$-cycle} is a sequence of $p$ distinct integers $\tau = (i_1, i_2, \ldots, i_p)\in [k]^p$.
When a matrix $Z \in \mathbb{R}^{k \times k}$ is further given, every $p$-cycle $\tau$ defines a product 
\begin{equation*}
    Z_{\tau} = \prod_{\ell=1}^{\rev{p}} Z_{i_\ell, i_{\ell+1}},
\end{equation*}
with the convention that $i_{\rev{p}+1} \equiv i_1$ for simpler notation. With this notation, we are ready to state one key idea towards the construction of $\KV$.

\begin{lemma}[{\cite{li2014sketching} and~\cite[Fact 2]{kong2017spectrum}}]\label{lem:unbiasedness}
For any $p$-cycle $\tau = (i_1, \ldots, i_{p})$ with $i_\ell \in [k]$, a real matrix $A \in \mathbb{R}^{m \times n}$, and an $n \times k$ random matrix $\Omega = (\omega_{j,\ell})$ where $\omega_{j, \ell}$ are i.i.d. entries with mean $0$ and variance $1$, 
\begin{equation*} 
\E[(\Omega^T A^T A \Omega)_{\tau}] = \mathrm{trace}((A^TA)^p) = \|A\|_{2p}^{2p},
\end{equation*}
where the expectation is over the randomness of $\Omega$.
\end{lemma}
\rev{\begin{proof}
By substituting the matrix $A$ with its singular value decomposition and leveraging the rotational invariance of the standard Gaussian distribution, we can assume without loss of generality that $A$ is diagonal and thus $A^T A = \diag(\sigma_1^2, \ldots, \sigma_n^2)$.
Let $\delta = (\delta_1, \ldots, \delta_p)\in [n]^p$. Then
\begin{equation*}
\left(\Omega^T A^T A \Omega\right)_\tau
= \prod_{h = 1}^p \left(\Omega^T A^T A \Omega\right)_{i_h, i_{h+1}}
= \prod_{h = 1}^p \left ( \sum_{t = 1}^n \sigma_t^2 \omega_{t, i_h} \omega_{t, i_{h+1}} \right )
= \sum_{\delta \in [n]^p} \prod_{h=1}^p \sigma_{\delta_h}^2 \omega_{\delta_h, i_h} \omega_{\delta_h, i_{h+1}}.
\end{equation*}
Since $\omega_{i,j}$'s are i.i.d. random variables with mean $0$ and variance $1$, the expectation of the product $\prod_{h=1}^p \sigma_{\delta_h}^2 \omega_{\delta_h, i_h} \omega_{\delta_h, i_{h+1}}$ vanishes unless every entry of $\Omega$ appears an even number of times in the product.
Since $i_1, \ldots, i_{p}$ are distinct, each non-zero term must have each $\omega_{\bullet, i_h}$ appear exactly twice and thus $\delta_h = \delta_{h+1}$. Therefore, the expectation of $\left(\Omega^T A^T A \Omega\right)_\tau$ becomes
\begin{equation*}
\E[(\Omega^T A^T A \Omega)_{\tau}] 
= \sum_{\delta_1 = 1}^n \prod_{h=1}^p \left( \E [\omega_{\delta_1, i_h}^2] \sigma_{\delta_1}^2 \right)
= \sum_{\delta_1 = 1}^n \sigma_{\delta_1}^{2p}
= \tr((A^T A)^p).\qedhere
\end{equation*}
\end{proof}}
Setting $Y = A \Omega$, \Cref{lem:unbiasedness} says that $(Y^T Y)_\tau$ is an unbiased estimate of $\|A\|_{2p}^{2p}$ for any (fixed) $p$-cycle $\tau$. A natural approach to reduce the variance of this estimator is to average over all possible $p$-cycles. However, this strategy poses significant computational challenges, as naively iterating over all $\frac{k!}{(k-p)!}$ possible cycles becomes computationally infeasible, even for moderately large values of $k$. A surprising linear algebra fact is that the average over all \emph{increasing} $p$-cycles, that is, $p$-cycles such that $i_1 < i_{\rev{2}}< \ldots < i_p$, can be efficiently computed by simple linear algebra operations \cite{kong2017spectrum}, and this defines $\KV$ presented in \Cref{alg:KV}:
\begin{equation} \label{eq:kv-def}
\binom{k}{p}^{-1} \sum_{1 \leq i_1 < \cdots < i_p \leq k} (\Omega^T A^T A \Omega)_{\tau} = \binom{k}{p}^{-1} \mathrm{trace}(T^{p-1}Z) \rev{= \KV},
\end{equation}
where $Z $ and $T$ are defined as in \Cref{alg:KV}. The proof for \rev{the first }identity in \eqref{eq:kv-def} can be found in \cite[Lemma 1]{kong2017spectrum}. We note that \Cref{lem:unbiasedness} together with the relation \eqref{eq:kv-def} guarantee the unbiasedness of $\KV$.

\section{Improved variance estimates for $\KV$}\label{sec:estimates}

This section presents our main results regarding the variance of $\KV$ and the analysis.
All the theorems are stated in \Cref{sec:main} while all the proofs are deferred to \Cref{sec:proofs}.

\subsection{Main results}\label{sec:main}
Our first result provides the first order expansion in $1/k$ for the variance of $\KV$ where $k$ is the sketch size. 

\begin{theorem}\label{thm:firstorder}
Let $\KV$ be the unbiased estimator for $\|A\|_{2p}^{2p}$ given in \Cref{alg:KV}. Then the variance of $\KV$ is
\begin{equation*}
    \mathrm{Var}(\KV) = \frac{2p^2 \|A\|_{4p}^{4p}}{k} + \mathcal O\left ( \frac{1}{k^2} \right ),
\end{equation*}
\rev{where} the constant in $\calO(\cdot)$ \rev{depends on $p$ and }is a symmetric function of the singular values of $A$.
\end{theorem}

The next theorem provides an upper bound for the variance of 
\begin{equation}\label{eq:Theta}
\rev{\X \coloneqq \binom{k}{p} \KV.}
\end{equation}
The bound is a function of $k$, $p$, and Schatten norms of $A$, highlighting lower variance for numerically low-rank matrices.

\begin{theorem}\label{thm:secondorderbound}
\rev{The variance of $\X$} is bounded by 
\begin{align*}
\var(\X) & \le - \binom{k}{p}^2\|A\|_{2p}^{4p} + \binom{k}{2p} \binom{2p}{p} \|A\|_{2p}^{4p}
+ \binom{k}{2p-1}\binom{2p-1}{p} p \left ( \|A\|_{2p}^{4p} + 2 \|A\|_{4p}^{4p} \right ) \\
&\quad + \binom{k}{2p-2}\binom{2p-2}{p}\binom{p}{2} \left ( 6 \|A\|_{4p}^{4p} + 3 \|A\|_4^4\|A\|_{4p-4}^{4p-4}\right ) \\
&\quad + \sum_{r=3}^p \binom{k}{2p-r}\binom{2p-r}{p}\binom{p}{r} \sum_{\ell = 1}^r c(r,\ell) \|A\|_4^{4(\ell-1)} \|A\|_{4p-4(\ell-1)}^{4p-4(\ell-1)},
\end{align*}
where
\begin{equation*}
    c(r,\ell) = \begin{cases}
        3^r & \text{if } \ell=1;\\
        3^{r-2}(2^{r+1}-1) & \text{if } \ell = 2;\\
        3^{r-\ell} \ell^r &\text{otherwise.}
    \end{cases}
\end{equation*}
\end{theorem}

Dividing the bound for $\var(\hat \Theta_{2p})$ in \Cref{thm:secondorderbound} by $\binom{k}{p}^2$ immediately provides an upper bound for $\var(\KV)$.  The second-order estimate for the variance of $\KV$ presented in \Cref{cor:secondorder} below follows from \Cref{thm:secondorderbound}.

\begin{corollary}\label{cor:secondorder}
The variance of $\KV$ satisfies
\begin{equation*}
\var(\KV) \le \frac{2p^2\|A\|_{4p}^{4p}}{k} + \frac{p^2 (p-1)^2}{k^2} \left (\|A\|_{4p}^{4p} + \frac{3}{2} \|A\|_4^4\|A\|_{4p-4}^{4p-4} - \frac{1}{2}\|A\|_{2p}^{4p} \right ) + \mathcal O \left ( \frac{1}{k^3} \right ),
\end{equation*}
\rev{where the constant in $\calO(\cdot)$ depends on $p$ and is a symmetric function of the singular values of $A$.}
\end{corollary}

\begin{remark}[\rev{A comparison with~\eqref{eq:varKV-original}}]
    \rev{It is not straightforward to compare the variance bound of $\hat \Theta_{2p}$ presented in Theorem~\ref{thm:secondorderbound} with the existing bound~\eqref{eq:varKV-original}, for a general matrix $A$, because different Schatten norms are involved in the bound of Theorem~\ref{thm:secondorderbound} and the ratios between different Schatten norms strongly depend on the singular value decay of the matrix. However, if the matrix $A$ is numerically low-rank, then all the Schatten norms are of comparable magnitude. If we assume, for simplicity, that all the Schatten norms are approximately $1$, we can further upper bound -- and simplify -- the result of Theorem~\ref{thm:secondorderbound}: Using the crude bounds $c(r,\ell) \le 3^p p^{p+1}$ and $\sum_{\ell=1}^r c(r,\ell) \le 3^p p^{p+2}$, we have
    }
    \rev{\begin{align}
        \mathrm{Var}(\KV) & \lesssim \left ( \frac{\binom{k}{2p} \binom{2p}{p}}{\binom{k}{p}^2} - 1\right ) + 3^p p^{p+2} \sum_{r=1}^p \frac{\binom{k}{2p-r} \binom{2p-r}{p} \binom{p}{r}}{\binom{k}{p}^2}\nonumber \\
        & \lesssim 0 + 3^p p^{p+2} \sum_{r=1}^p \frac{1}{k^r} \frac{p!p!}{(p-r)!r!(p-r)!}\le 3^p p^{p+2} p^{2p} \frac{p}{k} = 3^p p^{3p+3} \frac{1}{k},\label{eq:simplified}
    \end{align}
    where the second inequality holds approximately, since $\binom{k}{2p-r} \binom{k}{p}^{-2} \approx k^{-r} (p!)^2 / (2p-r)!$ when $k$ is much larger than $p$, and we totally ignore the denominator $((p-r)!)^2r!$ in the last line. Even with these crude approximations, the bound~\eqref{eq:simplified} compares favorably to~\eqref{eq:varKV-original}, suggesting that the bound of Theorem~\ref{thm:secondorderbound} improves~\eqref{eq:varKV-original} at least for numerically low-rank matrices. The numerical examples (\Cref{ex:identity,ex:varyingn}) in Section~\ref{sec:examples} show that Theorem~\ref{thm:secondorderbound} improves~\eqref{eq:varKV-original} also in cases with no singular value decay.}
\end{remark}

\begin{remark}[\rev{The constant in the $\mathcal O(\cdot)$}]
\rev{Both in Theorem~\ref{thm:firstorder} and Theorem~\ref{thm:secondorderbound}, the costant in the $\mathcal O(\cdot)$ is a function of $p$ and of the singular values of $A$.  While we refrain from attempting to write down an explicit constant, we will see in the numerical examples that this constant is moderately small if the matrix $A$ has singular values with strong singular value decay, which is the focus of our paper. In this case, the first-order and second-order estimates are good approximations for the variance for moderate values of $k$. This is not true when $A$ does not exhibit singular value decay: for instance, if we consider the identity matrix (see Example~\ref{ex:identity} below), the constant in the $\mathcal{O}(\cdot)$ actually leads to imprecise first-order and second-order estimates.}
\end{remark}

\begin{remark}[\rev{Rule of thumb for selecting $k$}]
\rev{ In the case of a numerically low-rank matrix $A$, when all Schatten norms are of comparable magnitude, Theorem~\ref{thm:firstorder} and Corollary~\ref{cor:secondorder} suggest that one needs to take $k$ to be proportional to $p^2$ in order to obtain a meaningful estimate for the $2p$-th power of the Schatten-$2p$ norm of $A$. }
\end{remark}

\begin{remark}[\rev{A comparison with Hutchinson trace estimation}]
\rev{When $p=1$, we have $\hat \theta_{2} = \frac{1}{k} \mathrm{trace}(\Omega^T A^T A \Omega)$, which coincides with the Hutchinson trace estimator applied to the matrix $A^TA$. When $p>1$, however, the two estimators differ and} the variance of $\KV$ is usually worse than the variance of estimators based on the Hutchinson trace estimator, though both are unbiased.
Specifically, since the variance of the vanilla Hutchinson trace estimator~\cite{Hutchinson1989} for a matrix $B$ with Gaussian random vectors is $2\|B\|_F^2$, applying it to~\eqref{eq:hutch} gives an unbiased estimator of $\|A\|_{2p}^{2p}$ with variance
\begin{equation*}
    \frac{2p}{k} \|A^{2p}\|_F^2 = \frac{2p}{k} \|A\|_{4p}^{4p},
\end{equation*}
which is smaller than \eqref{eq:varKV-ours}. We remark that, in order to use the Hutchinson trace estimator \rev{with $p>1$}, the \rev{matvecs with $A$ needs to be sequentially applied at least $p$ times}. When the only available information about the matrix $A$ is its sketch $Y$, to the best of the authors' knowledge, the estimator $\KV$ proposed in \cite{kong2017spectrum} is the only method to estimate the Schatten-$2p$ norms of $A$.
\end{remark}

\subsection{Proofs of the main results}\label{sec:proofs}

The proofs of our main results are separated into four subsections. 
\Cref{sec:breaking} breaks down the analysis for $\var(\KV)$ into several subtasks.
\Cref{sec:firstorder} presents the proof of first-order result (\Cref{thm:firstorder}). 
\Cref{sec:2-more-indices} and \Cref{sec:secondorder} are devoted for the upper bound (\Cref{thm:secondorderbound}) and second-order estimate (\Cref{cor:secondorder}).

\subsubsection{Breaking Var($\KV$) into mathematically manageable pieces}
\label{sec:breaking}

Let $A$ be an $m \times n$ deterministic matrix and $\Omega \in \mathbb{R}^{n \times k}$ a random matrix with i.i.d. $N(0,1)$ entries. Considering the full singular value decomposition $A = U \Sigma V^T$, where $U \in \mathbb{R}^{m \times m}$, $\Sigma \in \mathbb{R}^{m \times n}$, and $V \in \mathbb{R}^{n \times n}$, the entries of $\widetilde{\Omega} := V^T \Omega$ are also i.i.d. $N(0,1)$ random variables, due to the rotation invariance of the standard Gaussian distribution. \rev{Then
\begin{equation*}
    \Omega^T A^T A \Omega = \widetilde \Omega^T  \Sigma^T U^T U\Sigma \widetilde \Omega = \widetilde \Omega^T \Sigma^T \Sigma \widetilde \Omega.
\end{equation*}}
Denoting
\begin{equation*}
    \Sigma^T \Sigma = \begin{bmatrix} \sigma_1^2 & & & & \\ & \sigma_2^2 & & & \\ & & \ddots & & \\ & & & \ddots & \\ & & & & \sigma_n^2 \end{bmatrix} \quad \text{ and } \quad \widetilde \Omega = \begin{bmatrix} \omega_{1,1} & \omega_{1,2} & \cdots & \omega_{1,k} \\ \omega_{2,1} & \omega_{2,2} & \cdots & \omega_{2,k} \\ \vdots & \vdots & & \vdots \\ \vdots & \vdots & & \vdots \\ \omega_{n,1} & \omega_{n,2} & \cdots & \omega_{n,k} \end{bmatrix}
\end{equation*}
with the convention that $\sigma_{\min\{n,m\}+1} =\ldots =  \sigma_{\max\{n,m\}} = 0$ and using \rev{the first identity in \eqref{eq:kv-def}}, the estimator $\KV$ from \Cref{alg:KV} takes the form:
\begin{align*}
    \KV & = \rev{\binom{k}{p}^{-1} \sum_{i_1<\cdots<i_p} (\widetilde \Omega^T \Sigma^T \Sigma \widetilde \Omega)_{i_1i_2}(\widetilde \Omega^T \Sigma^T \Sigma \widetilde \Omega)_{i_2i_3}\cdots (\widetilde \Omega^T \Sigma^T \Sigma \widetilde\Omega)_{i_pi_1}}\\
    & = \binom{k}{p}^{-1} \sum_{i_1 < \cdots < i_p} (\sigma_1^2 \omega_{1, i_1}\omega_{1, i_2} + \cdots + \sigma_n^2 \omega_{n, i_1} \omega_{n, i_2}) \cdots (\sigma_1^2 \omega_{1, i_p}\omega_{1, i_1} + \cdots + \sigma_n^2 \omega_{n, i_p} \omega_{n, i_1}) \\
    & = \binom{k}{p}^{-1}  \sum_{i_1 < \cdots < i_p} \prod_{h = 1}^p \left ( \sum_{t = 1}^n \sigma_t^2 \omega_{t, i_h} \omega_{t, i_{h+1}}\right ),
 \end{align*}
with the indices $i_1, \ldots, i_p$ ranging from $1$ to $k$ and with the convention that $i_{p+1} \equiv i_1$. For ease of notation, \rev{we work with the rescaled quantity $\X$ defined in~\eqref{eq:Theta}. Then,}
\begin{equation}\label{eq:varKVX}
    \mathrm{Var}(\KV) = \binom{k}{p}^{-2}\left ( \mathbb{E}[\X^2] - \mathbb{E}[\X]^2 \right ) = \binom{k}{p}^{-2} \mathbb{E}[\X^2] - \|A\|_{2p}^{4p},
\end{equation}
where the second equality follows because $\KV$ is an unbiased estimator of $\|A\|_{2p}^{2p}$. 
It remains to estimate/bound the second moment of $\X$. Note that
\begin{equation}\label{eq:Xsquare}
\hat \Theta_{2p}^2 = \sum_{i_1 < \cdots < i_p \atop j_1 < \cdots < j_p} \prod_{h = 1}^p \left [ \left ( \sum_{t = 1}^n \sigma_t^2 \omega_{t, i_h} \omega_{t, i_{h+1}}\right ) \left ( \sum_{\rev{s} = 1}^n \sigma_{\rev{s}}^2 \omega_{\rev{s}, j_h} \omega_{\rev{s}, j_{h+1}}\right ) \right ],
\end{equation}
where the indices $i_1, \ldots, i_p$ and $j_1, \ldots, j_p$ range from $1$ to $k$. 
Estimating or bounding the variance of $\rev{\hat \Theta_{2p}}$ involves estimating or bounding the expectation of each term in the outer sum from~\eqref{eq:Xsquare}. For ease of notation, we denote
\begin{equation}\label{eq:f}
    f_{i_1,\ldots,j_p} := \prod_{h = 1}^p \left [ \left ( \sum_{t = 1}^n \sigma_t^2 \omega_{t, i_h} \omega_{t, i_{h+1}}\right ) \left ( \sum_{\rev{s} = 1}^n \sigma_{\rev{s}}^2 \omega_{\rev{s}, j_h} \omega_{\rev{s}, j_{h+1}}\right ) \right ];
\end{equation}
\rev{we define 
\begin{equation*}
    \mathcal I_r := \left \{ 2p\text{-tuples } i_1 < \cdots <i_p,j_1 <\cdots <j_p \text{ among which we have }r\text{ repeated indices}\right \}
\end{equation*}
and let 
\begin{equation} \label{eq:Fr}
F_r := \sum_{(i_1, \ldots, j_p) \in \mathcal I_r} \mathbb{E}[f_{i_1,\ldots,j_p}].
\end{equation}
be the sum of $\E[f_{i_1,\ldots,j_p}]$ over the choices of $i_1 < \cdots < i_p$ and $j_1 < \cdots < j_p$ with $r$ repeated indices.}
Then, we can split the second moment of $\X$ as
\begin{equation}\label{eq:splitXsquare}
\mathbb{E}[\hat \Theta_{2p}^2] = \sum_{r=0}^p \left( \sum_{\rev{(i_1,\ldots,j_p) \in \mathcal I_r}} \mathbb{E}[f_{i_1,\ldots,j_p}] \right) = \sum_{r=0}^p F_r.
\end{equation}
Our analysis follows by estimating or bounding each $F_r$.
\Cref{lemma:rrepeated,lem:degreeofk} below describes the contribution of each $F_r$ to $\E[\hat \Theta_{2p}^2]$ in terms of the order of $k$.

\begin{lemma}\label{lemma:rrepeated}
The \rev{cardinality of the set $\mathcal I_r$, for 
$0 \le r \le p$,} is $\binom{k}{2p-r} \binom{2p-r}{p} \binom{p}{r}$.
\end{lemma}

\begin{proof}
    The set $\{i_1, \ldots, j_p\}$ contains $2p-r$ distinct values from the set $\{1, \ldots, k\}$, so there are $\binom{k}{2p-r}$ possibilities for this. There are now $\binom{2p-r}{p}$ ways of assigning $i_1, \ldots, i_p$ to some of the chosen indices. Finally, there are $\binom{p}{r}$ ways of choosing which of these $p$ indices are repeated among $j_1, \ldots, j_p$. Once we made the decision, the assignment is uniquely determined by the fact that $i_1 < \cdots < i_p$ and $j_1 < \cdots < j_p$.
\end{proof}

\begin{lemma}\label{lem:degreeofk}
For each $0 \leq r \leq p$, the quantity $F_r$ defined in \eqref{eq:Fr}, as a polynomial in $k$, has degree at most $\mathcal O(k^{2p-r})$.
Hence, the second moment $\E[\hat \Theta_{2p}^2]$, as a polynomial in $k$, has degree at most $2p$.
\end{lemma}
\begin{proof}
Since the distribution of $\omega_{i,j}$ is independent of $k$ for any index $(i,j)$, so is the expectation $\mathbb{E}[f_{i_1,\ldots,j_p}]$.
By \Cref{lemma:rrepeated}, $F_r$ grows (as a function in $k$) at the same rate as $\binom{k}{2p-r} \binom{2p-r}{p} \binom{p}{r}$, which is a polynomial in $k$ of degree $2p-r$.
\end{proof}

\subsubsection{First-order analysis}\label{sec:firstorder}
Recall from \eqref{eq:varKVX} and \eqref{eq:splitXsquare}, the variance of $\KV$ is 
\begin{equation} \label{eq:splitFr}
\mathrm{Var}(\KV) 
= \binom{k}{p}^{-2} \mathbb{E}[\hat \Theta_{2p}^2] - \|A\|_{2p}^{4p} 
= \binom{k}{p}^{-2} \sum_{r=0}^p F_r - \|A\|_{2p}^{4p}.
\end{equation}
Since $\binom{k}{p}^2$ is a polynomial (in $k$) of degree $2p$ and, by \Cref{lem:degreeofk}, $F_r$ is a polynomial in $k$ of degree $2p-r$, it suffices to consider $F_0$ and $F_1$ in the first-order analysis.
Therefore, we are interested in the sum of $\E[f_{i_1,\ldots,j_p}]$ over the choices of \rev{$2p$-tuples $i_1 < \cdots < i_p$ and $j_1 < \cdots < j_p$ in either $\mathcal I_0$ or $\mathcal I_1$}. 
Expanding the product of the two summations in \eqref{eq:f}, 
the expectation $\E[f_{i_1,\ldots,j_p}]$ can be further decomposed into the sum of $n^{2p}$ expectations:
\begin{equation} \label{eq:f-to-sum-Y}
\mathbb{E}[f_{i_1,\ldots,j_p}] = \sum_{(t_1, \ldots,s_p) \in [n]^{2p}} \mathbb{E}\left[ \Y \right],
\end{equation}
where the random variable $\Y$ takes the form
\begin{equation}\label{eq:singleterm}
\Y :=\prod_{h=1}^p \sigma_{t_h}^2 \omega_{t_h,i_h} \omega_{t_h,i_{h+1}} \sigma_{s_h}^2 \omega_{s_h,j_h} \omega_{s_h,j _{h+1}}.
\end{equation}
Observe that, if any of the $\omega$-variables appearing in~\eqref{eq:singleterm} has an odd exponent, the expectation of $\Y$ becomes zero since all odd moments of standard Gaussian random variables are zero. This observation gives the following lemma.
\begin{lemma} \label{lem:nonrepeated-ih}
Suppose $i_h$ (resp. $j_h$), $h \in \{1, 2, \ldots, p\}$, is a non-repeated index among $i_1 < \cdots < i_p$ and $j_1 < \cdots < j_p$.
Then a necessary condition for $\E[\Y]$ to be nonzero is that $t_{h-1} = t_h$ (resp. $s_{h-1}=s_h$), with the convention that $t_0 \equiv t_p$ and $s_0 \equiv s_p$ for simpler notation.
\end{lemma}
\begin{proof}
The product of the random variables in \eqref{eq:singleterm} is
\begin{equation} \label{eq:omega-prod}
\omega_{t_1,i_1}\omega_{t_1,i_2}\omega_{t_2,i_2} \omega_{t_2,i_3} \cdots \omega_{t_p,i_p} \omega_{t_p,i_1} \omega_{s_1,j_1}\omega_{s_1,j_2}\omega_{s_2,j_2} \omega_{s_2,j_3} \cdots \omega_{s_p,j_p} \omega_{s_p,j_1}.
\end{equation}
If $\Y$ has nonzero expectation, each $\omega$-variable in \eqref{eq:omega-prod} must have even exponent. 
In particular, there must be a second copy of $\omega_{t_h, i_h}$. This can happen only if $\omega_{t_h, i_h} = \omega_{t_{h-1}, i_h}$ since $i_h$ is different from any other index in $\{i_1, \ldots, j_p\} \setminus \{i_h\}$ (\ie $i_h$ is non-repeated index). Hence, $t_{h-1} = t_h$.
The argument for the $s$-indices is the same.
\end{proof}

Leveraging \Cref{lem:nonrepeated-ih}, we can compute $\E[f_{i_1, \ldots, j_p}]$ in the cases where there are either no repeated indices or only one repeated index among $i_1, \ldots, j_p$. 

\begin{lemma}\label{lemma:r0}
If \rev{$(i_1, \ldots, j_p) \in \mathcal I_0$} 
then $$\mathbb{E}[f_{i_1,\ldots,j_p}] = \|A\|_{2p}^{4p}.$$
\end{lemma}

\begin{proof}
Consider a single nonzero term $\E[\Y]$ on the right-hand-side of \eqref{eq:f-to-sum-Y}. Then $t_1 = \ldots = t_p$ and $s_1 = \ldots = s_p$ by \Cref{lem:nonrepeated-ih} since there is no repeated index among $i_1 < \cdots < i_p$ and $j_1 < \cdots < j_p$. In this case, we have
\begin{equation*}
    \mathbb{E}[\Y] = \mathbb{E} \left [ \sigma_{t_1}^{2p} \sigma_{s_1}^{2p} \prod_{h=1}^p \omega_{t_1,i_h}^2 \omega_{s_1,j_h}^2 \right ] = \sigma_{t_1}^{2p} \sigma_{s_1}^{2p},
\end{equation*}
where the last equality follows from the fact that the $\omega_{t_1,i_h}$ and $\omega_{s_1,j_h}$ are independent random variables since there is no repeated index among $i_1, \ldots, j_p$. To obtain $\mathbb{E}[f_{i_1,\ldots,j_p}]$ we need to sum over all possible choices of $t_1$ and $s_1$, each of which can be anything in $\{1,\ldots,n\}$. Therefore,
\begin{equation*}
    \mathbb{E}[f_{i_1,\ldots,j_p}] = \sum_{t=1}^n \sum_{s=1}^n \sigma_t^{2p} \sigma_s^{2p} = \|A\|_{2p}^{4p}. \qedhere
\end{equation*}
\end{proof}

\begin{lemma}\label{lemma:r1}
If \rev{$(i_1, \ldots, j_p) \in \mathcal I_1$} 
then $$\mathbb{E}[f_{i_1,\ldots,j_p}] = \|A\|_{2p}^{4p} + 2\|A\|_{4p}^{4p} .$$
\end{lemma}

\begin{proof}
Let $\E[\Y]$ be a nonzero term on the right-hand-side of \eqref{eq:f-to-sum-Y}. 
Without loss of generality, let us assume that the repeated index is $i_1$. Then \Cref{lem:nonrepeated-ih} guarantees that $t_h = t_{h-1}$ for $h=2,\ldots,p$, so that $t_1 = \ldots = t_p$. The same argument also proves $s_1 = \ldots = s_p$. 
    
Now, given $t_1 = \ldots = t_p$ and $s_1 = \ldots = s_p$, we consider two disjoint subcases:
\begin{itemize}
\item If $t_1 \neq s_1$, then
\begin{equation*}
    \mathbb{E}[\Y] = \mathbb{E} \left [ \sigma_{t_1}^{2p} \sigma_{s_1}^{2p} \prod_{h=1}^p \omega_{t_1,i_h}^2 \omega_{s_1,j_h}^2 \right ] = \sigma_{t_1}^{2p} \sigma_{s_1}^{2p},
\end{equation*}
where the last equality follows from the fact that the $\omega_{t_1,i_h}$ and $\omega_{s_1,j_h}$ are independent random variables since $t_1 \neq s_1$.
\item If $t_1 = s_1$, then the fact that $i_1, \ldots, i_p$ and $j_1, \ldots, j_p$ have exactly one repeated index means that the product \eqref{eq:singleterm} has exactly one $N(0,1)$ random variable to the fourth power and all the others are squares of $N(0,1)$ random variables. Therefore,
\begin{equation*}
    \mathbb{E}[\Y] = \mathbb{E} \left [ \sigma_{t_1}^{2p} \sigma_{t_1}^{2p} \prod_{h=1}^p \omega_{t_1,i_h}^2 \omega_{t_1,j_h}^2 \right ] = 3 \sigma_{t_1}^{4p},
\end{equation*}
where the factor of $3$ comes from the single fourth moment of $N(0,1)$.
\end{itemize}
To obtain $\mathbb{E}[f_{i_1,\ldots,j_p}]$, we sum over all possible choices of $t_1$ and $s_1$, each of which can be anything in $\{1,\ldots,n\}$. Therefore,
\begin{equation*}
    \mathbb{E}[f_{i_1,\ldots,j_p}] = 3 \sum_{t=1}^n \sigma_t^{4p} + \sum_{t\neq s} \sigma_t^{2p} \sigma_s^{2p} = 3\|A\|_{4p}^{4p} + \left ( \|A\|_{2p}^{4p} - \|A\|_{4p}^{4p} \right ) = \|A\|_{2p}^{4p} + 2\|A\|_{4p}^{4p}. \qedhere
\end{equation*}
\end{proof}

Given \Cref{lemma:r0,lemma:r1}, we are ready to compute the first-order expansion of $\var(\KV)$ in $1/k$ presented in \Cref{thm:firstorder}.

\begin{proof}[Proof of \Cref{thm:firstorder}]
Combining \Cref{lemma:rrepeated,lemma:r0,lemma:r1}, we obtain
\begin{align*}
F_0 &= \binom{k}{2p} \binom{2p}{p} \|A\|_{2p}^{4p} = \frac{k^{2p} - p(2p-1)k^{2p-1}}{(p!)^2} \|A\|_{2p}^{4p} + \calO(k^{2p-2}), \\
F_1 &= \binom{k}{2p-1} \binom{2p-1}{p} p (\|A\|_{2p}^{4p} + 2 \|A\|_{4p}^{4p}) \\
&= \frac{k^{2p-1}}{[(p-1)!]^2} (\|A\|_{2p}^{4p} + 2 \|A\|_{4p}^{4p}) + \calO(k^{2p-2}).
\end{align*}
All the constants in $\calO(\cdot)$ involve only $p$, $\|A\|_{2p}$, and $\|A\|_{4p}$, so they are symmetric functions in singular values of $A$.
Moreover, \Cref{lem:degreeofk} guarantees
\begin{equation*}
\sum_{r=2}^p F_r = \calO(k^{2p-2})
\end{equation*}
and the constants in $\calO(\cdot)$ are again symmetric functions in singular values of $A$ since each $\E[f_{i_1, \ldots, j_p}]$ is. 
Combining all the above, the variance of $\KV$ becomes 
\begin{align*}
\mathrm{Var}(\KV) &= \binom{k}{p}^{-2} \mathbb{E}[\X^2] - \|A\|_{2p}^{4p} \\
&= \binom{k}{p}^{-2} \left[ F_0 + F_1 + \sum_{r=2}^p F_r - \binom{k}{p}^2 \|A\|_{2p}^{4p} \right] \\
&= \frac{(p!)^2}{k^2 (k-1)^2 \cdots (k-p+1)^2} \left[ \frac{k^{2p} - p(2p-1)k^{2p-1}}{(p!)^2} \|A\|_{2p}^{4p} \right. \\[5pt]
&\left. \qquad + \frac{k^{2p-1}}{[(p-1)!]^2} (\|A\|_{2p}^{4p} + 2 \|A\|_{4p}^{4p}) - \frac{k^{2p} - p(p-1)k^{2p-1}}{(p!)^2} \|A\|_{2p}^{4p} + \calO(k^{2p-2}) \right]  \\[5pt]
&= \frac{2 p^2 \|A\|_{2p}^{4p}}{k} + \calO\left ( \frac{1}{k^2} \right ),
\end{align*}
where all the constants inside $\calO(\cdot)$ are symmetric functions of singular values of $A$.
\end{proof}

\subsubsection{Terms with at least 2 repeated indices}\label{sec:2-more-indices}
To obtain an upper bound for $\var(\KV)$, we need to further consider \rev{$2p$-tuples $(i_1, \ldots, j_p) \in \mathcal I_r$ with } 
$r \geq 2$.
Motivated by the decomposition in \eqref{eq:f-to-sum-Y}, we essentially want to count or bound the number of tuples $(t_1, \ldots, s_p) \in [n]^{2p}$ such that $\Y$ has nonzero expectation. 
To better visualize the product
\begin{equation} \label{eq:Y-prod}
\Y = \prod_{h=1}^p \sigma_{t_h}^2 \omega_{t_h,i_h} \omega_{t_h,i_{h+1}} \sigma_{s_h}^2 \omega_{s_h,j_h} \omega_{s_h,j _{h+1}}
\end{equation}
for some fixed $i_1, \ldots, i_p,j_1, \ldots, j_p$, 
let us consider two circles with $p$ points each, corresponding to $i_1, \ldots, i_p$ and $j_1, \ldots, j_p$, respectively. This is illustrated in Figure~\ref{fig:circles1} for $p=8$. The $r$ repeated indices are denoted with an extra circle; in Figure~\ref{fig:circles1} we have 3 repeated indices, that is, $i_1 = j_2$, $i_2 = j_4$, and $i_5 = j_8$. Note that such pairings between repeated indices are unique since $i_1 < \cdots < i_p$ and $j_1 < \cdots < j_p$.
Fixing the indices $i_1, \ldots, j_p$, a single choice of $(t_1, \ldots, s_p) \in [n]^{2p}$ for $\Y$ corresponds to assigning singular value $\sigma_{t_h}$ to the segment $(i_h, i_{h+1})$, and assigning singular value $\sigma_{s_h}$ to the segment $(j_h, j_{h+1})$.

\begin{figure}[htb]
\centering
\begin{tikzpicture}
\def\radius{2cm}
\definecolor{color1}{HTML}{FF0000} 
\definecolor{color2}{HTML}{FFA500} 
\definecolor{color3}{HTML}{FFFF00} 
\definecolor{color4}{HTML}{00FF00} 
\definecolor{color5}{HTML}{00FFFF} 
\definecolor{color6}{HTML}{0000FF} 
\definecolor{color7}{HTML}{800080} 
\definecolor{color8}{HTML}{FFC0CB} 
\draw (0,0) circle (\radius);
\foreach \i in {1,...,8} {
    \coordinate (P\i) at (\i*45-45: \radius);
    \fill (P\i) circle (2pt);
    \node[anchor=\i*45-45+180] at (P\i) {$i_{\i}$};
}
\draw (P1) circle (4pt);
\draw (P2) circle (4pt);
\draw (P5) circle (4pt);
\draw[very thick, color1] (P1) arc[start angle=0, end angle=45, radius=\radius];
\draw[very thick, color2] (P2) arc[start angle=45, end angle=180, radius=\radius];
\draw[very thick, color4] (P5) arc[start angle=180, end angle=360, radius=\radius];
\end{tikzpicture}
\begin{tikzpicture}
\def\radius{2cm}
\definecolor{color1}{HTML}{FF0000} 
\definecolor{color2}{HTML}{FFA500} 
\definecolor{color3}{HTML}{FFFF00} 
\definecolor{color4}{HTML}{00FF00} 
\definecolor{color5}{HTML}{00FFFF} 
\definecolor{color6}{HTML}{0000FF} 
\definecolor{color7}{HTML}{800080} 
\definecolor{color8}{HTML}{FFC0CB} 
\draw (0,0) circle (\radius);
\foreach \i in {1,...,8} {
    \coordinate (P\i) at (\i*45-45: \radius);
    \fill (P\i) circle (2pt);
    \node[anchor=\i*45-45+180] at (P\i) {$j_{\i}$};
}
\draw (P2) circle (4pt);
\draw (P4) circle (4pt);
\draw (P8) circle (4pt);
\draw[very thick, color3] (P2) arc[start angle=45, end angle=135, radius=\radius];
\draw[very thick, color5] (P4) arc[start angle=135, end angle=315, radius=\radius];
\draw[very thick, color6] (P8) arc[start angle=315, end angle=405, radius=\radius];
\end{tikzpicture}
\caption{Circle notation. }
\label{fig:circles1}
\end{figure}

Let $i_h$ be a non-repeated index (corresponding to non-circled point in \Cref{fig:circles1}). 
\Cref{lem:nonrepeated-ih} asserts that the two segments $(i_{h-1}, i_{h})$ and $(i_h, i_{h+1})$ adjacent to point $i_h$ must be assigned to the same singular value; otherwise the expectation of $\Y$ will vanish.
To illustrate this fact, in \Cref{fig:circles1}, we used the same color to highlight the segments that need to be assigned to the same singular value as a necessary condition for the expectation of the term~\eqref{eq:singleterm} to be nonzero. In the following, we refer to the group of contiguous segments with the same color as an ``arc''.

Consider a \rev{$2p$-tuple $(i_1, \ldots, j_p) \in \mathcal I_r$.} 
Then, each circle has $r$ colored arcs. Moreover, there is a natural bijection from the colored arcs in $i$-circle to the color arcs in $j$-circle, where each colored arcs is mapped to the colored arc in the other circle with the same indices as endpoints.
In our example, the pairings are $(i_1,i_2)$--$(j_2,j_4)$, $(i_2,i_5)$--$(j_4,j_8)$, and $(i_5,i_1)$--$(j_8,j_2)$, where each arc is intended in a counter-clockwise direction. We denote the arcs by $a_1^{(i)}, \ldots, a_r^{(i)}, a_1^{(j)}, \ldots, a_r^{(j)}$ and we will say that $a_\ell^{(i)}$ and $a_\ell^{(j)}$ are \emph{corresponding pairs of arcs} for $\ell = 1, \ldots, r$.

The next lemma limits the possible singular values assignment to $2r$ arcs in the two circles and is crucial for the proofs of \Cref{lemma:r2} and \Cref{lemma:rlarger}.

\begin{lemma} \label{lem:obs}
Fix a \rev{$2p$-tuple $(i_1, \ldots, j_p) \in \mathcal I_r$} 
and a $2p$-tuple $(t_1, \ldots, s_p) \in [n]^{2p}$ such that $\E[\Y]$ is nonzero. Consider the visualization for $\Y$ as in \Cref{fig:circles1}.
Suppose the arcs $a_{1}^{(i)}$ and $a_{2}^{(i)}$ in the $i$-circle are assigned to singular values $\sigma_g$ and $\sigma_h$, respectively, with $g \neq h$.
Then there are only two possible singular values assignments to their corresponding arcs $a_{1}^{(j)}$ and $a_{2}^{(j)}$ in the $j$-circle:
\begin{enumerate}[(1)]
\item\label{item:case1} $a_{1}^{(j)}$ is assigned to $\sigma_g$ and $a_{2}^{(j)}$ is assigned to $\sigma_h$. 
Moreover, if this is the case, then for each $\ell = 1, \ldots, r$, the singular values assigned to corresponding pair of arcs $a_{\ell}^{(i)}$ and $a_{\ell}^{(j)}$ must be the same. (Singular values assigned to different pairs could be different.)

\item\label{item:case2} $a_{1}^{(j)}$ is assigned to $\sigma_h$ and $a_{2}^{(j)}$ is assigned to $\sigma_g$. 
Moreover, if this is the case, then every arc $a_{\ell}^{(i)}$ in $i$-circle must be assigned to either $\sigma_g$ or $\sigma_h$; while the corresponding arc $a_{\ell}^{j}$ in $j$-circle must be assigned to the other singular value. That is, no singular values other than $\sigma_g$ and $\sigma_h$ can be present in this case.
\end{enumerate}
\end{lemma}
\begin{proof}
Let $i^{\ast}_2$ denote the common repeated index (common endpoint) for arcs $a_{1}^{(i)}$ and $a_{2}^{(i)}$. Note that $i^{\ast}_2$ is also the common endpoint of $a_{1}^{(j)}$ and $a_{2}^{(j)}$.
Since $a_{1}^{(i)}$ is assigned to $\sigma_g$ and $a_{2}^{(i)}$ is assigned to $\sigma_h$, the product of the $\omega$-variables in \eqref{eq:Y-prod} contains $\omega_{g, i_2^{\ast}}$ and $\omega_{h, i_2^{\ast}}$, which are different random variables since $g \neq h$. 
To have nonzero $\E[\Y]$, there must be a second copy of both $\omega_{g, i_2^{\ast}}$ and $\omega_{h, i_2^{\ast}}$, and they must come from the segments joining the common endpoint of $a_{1}^{(j)}$ and $a_{2}^{(j)}$ in $j$-circle since the repeated index $i_2^{\ast}$ does not appear in the endpoint of any other segment in two circles. This means that one of the arc $a_{1}^{(j)}$ and $a_{2}^{(j)}$ has to be assigned to $\sigma_g$ and the other has to be assigned to $\sigma_h$, which proves the fact that there are only two possible singular values assignment to $a_{1}^{(j)}$ and $a_{2}^{(j)}$ in $j$-circle.

Now, consider case \ref{item:case1} and case \ref{item:case2} separately. Denote $i_3^{\ast}$ as the repeated index commonly appear at the endpoints of $a_{2}^{(i)}$, $a_2^{(j)}$, $a_{3}^{(i)}$, and $a_3^{(j)}$.
\begin{enumerate}[(1)]
\item In this case, both $a_{2}^{(i)}$ and $a_2^{(j)}$ are assigned to $\sigma_h$. Such assignment contributes $\omega_{h,i_3^{\ast}}^2$ in the product of \eqref{eq:Y-prod}. To guarantee that the other two $\omega_{\bullet, i_3^{\ast}}$ from endpoint $i_3^{\ast}$ of $a_{3}^{(i)}$ and $a_3^{(j)}$ coincide (so that $\E[\Y]$ is nonzero), the corresponding pair of arcs $a_{3}^{(i)}$ and $a_3^{(j)}$ must be assigned to the same singular value. Inductively, one can prove that, for every $\ell = 1, \ldots, r$, the corresponding arcs $a_{\ell}^{(i)}$ and $a_{\ell}^{(j)}$ must be assigned to the same singular value.

\item In this case, $a_{2}^{(i)}$ is assigned to $\sigma_h$ and $a_2^{(j)}$ is assigned to $\sigma_g$. This assignment contributes $\omega_{h,i_3^{\ast}} \omega_{g,i_3^{\ast}}$ in the product of \eqref{eq:Y-prod}. To guarantee both $\omega_{h,i_3^{\ast}}$ and $\omega_{h,i_3^{\ast}}$ have even exponent in the product of \eqref{eq:Y-prod}, one of the arc $a_{3}^{(i)}$ and $a_{3}^{(j)}$ has to be assigned to $\sigma_g$ and the other has to be assigned to $\sigma_h$. The proof is complete by induction. \qedhere
\end{enumerate}
\end{proof}

Let us introduce one more notation which will be used later in our proofs. We assume each segment in the circles of \Cref{fig:circles1} has unit length.
Denote $d_{\ell}^{(i)}$ the length of arc $a_{\ell}^{(i)}$ and $d_{\ell}^{(j)}$ the length of arc $a_{\ell}^{(j)}$, for $\ell=1,\ldots,r$. All lengths are at least $1$ and we have
\begin{equation*}
\sum_{\ell=1}^r d_{\ell}^{(i)} = \sum_{\ell=1}^r d_{\ell}^{(j)} = p.
\end{equation*}
Our analysis will also use the following inequality between products of Schatten norms.

\begin{lemma}\label{lemma:inequalities}
Let $2 \le c \le d$ be two integers. Then $\|A\|_c^c \|A\|_d^d \le \|A\|_{c-1}^{c-1} \|A\|_{d+1}^{d+1}$.
\end{lemma}
\begin{proof}
    We need to show that, for any sequence of nonnegative integers $\sigma_1, \ldots, \sigma_n$, 
    \begin{equation*}
        \sum \sigma_i^c \sum \sigma_j^d \le \sum \sigma_i^{c-1} \sum \sigma_j^{d+1}.
    \end{equation*}
    Note that there is a bijection between the terms $\sigma_i^c\sigma_j^d$ on the left-hand-side and the terms $\sigma_i^{c-1} \sigma_j^{d+1}$ on the right-hand-side. Therefore, it is sufficient to show that, for all $i, j \in \{1, \ldots, n\}$, we have
    \begin{equation}\label{eq:wts1}
        \sigma_i^c\sigma_j^d + \sigma_i^d \sigma_j^c \le \sigma_i^{c-1}\sigma_j^{d+1} + \sigma_i^{d+1}\sigma_j^{c-1}.
    \end{equation}
    We have that
    \begin{equation}\label{eq:factorization}
        \sigma_i^{c-1}\sigma_j^{d+1} - \sigma_i^c\sigma_j^d - \sigma_i^d \sigma_j^c + \sigma_i^{d+1}\sigma_j^{c-1} = \sigma_i^{c-1}\sigma_j^{c-1} (\sigma_j - \sigma_i)\left ( \sigma_j^{d-c+1} - \sigma_i^{d-c+1}\right ).
    \end{equation}
    Since $d\ge c$, we also have that $d-c+1 \ge 0$, therefore $\sigma_j^{d-c+1} - \sigma_i^{d-c+1}$ has the same sign as $\sigma_j - \sigma_i$. Hence,~\eqref{eq:factorization} is always nonnegative, which implies~\eqref{eq:wts1} by rearranging the terms, and therefore implies the result of the lemma.
\end{proof}

Now, we are ready to bound $\E[f_{i_1, \ldots, j_p}]$ when \rev{the $2p$-tuple $(i_1,\ldots,j_p) \in \mathcal I_r$ with $r \ge 2$.} 

\begin{lemma}\label{lemma:r2}
\rev{If $(i_1,\ldots,j_p) \in \mathcal I_2$}, then $$\mathbb{E}[f_{i_1,\ldots,j_p}] \le 6\|A\|_{4p}^{4p} + 3\|A\|_4^4\|A\|_{4p-4}^{4p-4}.$$
\end{lemma}

\begin{proof}
    Since there are two repeated indices, there are two arcs for each circle. Using our previous notation, they are $a_1^{(i)}$, $a_2^{(i)}$, $a_1^{(j)}$, and $a_2^{(j)}$. 
    By \Cref{lem:obs}, the nonzero expectation $\E[\Y]$ on the right-hand-side of \eqref{eq:f-to-sum-Y} corresponds to one of the following situations. 
    \begin{itemize}
        \item All arcs have the same singular value:
        \begin{equation*}
            \mathbb{E}[\Y] = \mathbb{E} \left [ \sigma_{t_1}^{2p} \sigma_{t_1}^{2p} \prod_{h=1}^p \omega_{t_1,i_h}^2 \omega_{t_1,j_h}^2 \right ] = 9 \sigma_{t_1}^{4p},
        \end{equation*}
        where the factor $9$ comes from the fact that two indices are repeated and therefore there are two fourth powers of i.i.d. $N(0,1)$ random variables.

        \item The arcs in the $i$-circle have one singular value $\sigma_{t_1}$ and the arcs in the $j$-circle have another singular value $\sigma_{s_1}$:
        \begin{equation*}
            \mathbb{E}[\Y] = \mathbb{E} \left [ \sigma_{t_1}^{2p} \sigma_{s_1}^{2p} \prod_{h=1}^p \omega_{t_1,i_h}^2 \omega_{s_1,j_h}^2 \right ] = \sigma_{t_1}^{2p} \sigma_{s_1}^{2p}.
        \end{equation*}

        \item Arcs $a_1^{(i)}$ and $a_1^{(j)}$ have the same singular value $\sigma_{t_1}$, and arcs $a_2^{(i)}$ and $a_2^{(j)}$ have the same singular value $\sigma_{t_2}$, with $t_1 \neq t_2$:
        \begin{equation*}
            \mathbb{E}[\Y] = \mathbb{E} \left [ \sigma_{t_1}^{2(d_1^{(i)} + d_1^{(j)})} \sigma_{t_2}^{2(d_2^{(i)} + d_2^{(j)})} \prod_{h=1}^p \omega_{t_1,i_h}^2 \omega_{s_1,j_h}^2 \right ] = \sigma_{t_1}^{2(d_1^{(i)} + d_1^{(j)})} \sigma_{t_2}^{2(d_2^{(i)} + d_2^{(j)})}.
        \end{equation*}

        \item Arcs $a_1^{(i)}$ and $a_2^{(j)}$ have the same singular value $\sigma_{t_1}$, and arcs $a_2^{(i)}$ and $a_1^{(j)}$ have the same singular value $\sigma_{t_2}$, with $t_1 \neq t_2$:
        \begin{equation*}
            \mathbb{E}[\Y] = \mathbb{E} \left [ \sigma_{t_1}^{2(d_1^{(i)} + d_2^{(j)})} \sigma_{t_2}^{2(d_2^{(i)} + d_1^{(j)})} \prod_{h=1}^p \omega_{t_1,i_h}^2 \omega_{s_1,j_h}^2 \right ] = \sigma_{t_1}^{2(d_1^{(i)} + d_2^{(j)})} \sigma_{t_2}^{2(d_2^{(i)} + d_1^{(j)})}.
        \end{equation*}
    \end{itemize}
    Now we need to sum up all these contributions:
    \begin{align*}
        \mathbb{E}[f_{i_1,\ldots,j_p}] & = 9\sum_{t=1}^n \sigma_t^{4p} + \sum_{t \neq s} \left ( \sigma_t^{2p} \sigma_s^{2p} + \sigma_{t}^{2(d_1^{(i)} + d_1^{(j)})} \sigma_{s}^{2(d_2^{(i)} + d_2^{(j)})} + \sigma_{t}^{2(d_1^{(i)} + d_2^{(j)})} \sigma_{s}^{2(d_2^{(i)} + d_1^{(j)})}\right )\\
        & = 9 \|A\|_{4p}^{4p} + \left ( \|A\|_{2p}^{4p} + \|A\|_{2(d_1^{(i)} + d_1^{(j)})}^{2(d_1^{(i)} + d_1^{(j)})} + \|A\|_{2(d_1^{(i)} + d_2^{(j)})}^{2(d_1^{(i)} + d_2^{(j)})} - 3\|A\|_{4p}^{4p}\right )\\
        & \le 6\|A\|_{4p}^{4p} + 3 \|A\|_4^4 \|A\|_{4p-4}^{4p-4}  ,
    \end{align*}
    where we used Lemma~\ref{lemma:inequalities} (repeatedly) for the inequality at the end, since $2(d_1^{(i)} + d_1^{(j)}) \ge 4$ and $2(d_1^{(i)} + d_2^{(j)}) \ge 4$.
\end{proof}

\begin{lemma}\label{lemma:rlarger}
\rev{If $(i_1, \ldots, j_p) \in \mathcal I_r$ for some $r > 2$, }
then $$\mathbb{E}[f_{i_1,\ldots,j_p}] \le \sum_{\ell=1}^r c(r, \ell) \|A\|_{4}^{4(\ell-1)} \|A\|_{4p-4(\ell-1)}^{4p-4(\ell-1)}.$$
where
\begin{equation}\label{eq:defc}
        c(r,\ell) = \begin{cases}
            3^r & \text{if } \ell=1;\\
            3^{r-2}(2^{r+1}-1) & \text{if } \ell = 2;\\
            3^{r-\ell} \ell^r &\text{otherwise.}
        \end{cases}
    \end{equation}
\end{lemma}

\begin{proof}
It suffices to consider the terms $\Y$ with nonzero expectation.
By \Cref{lem:obs}, the maximum number of different colors in the two circles of \Cref{fig:circles1} (\ie the maximum number of different singular values in the product \eqref{eq:Y-prod}) is $r$.
We separately consider the terms $\Y$ in which there are $\ell$ distinct colors (singular values), for $1 \leq \ell \leq r$.
\begin{itemize}
    \item {\bf Case $\ell = 1$.} All the arcs in the two circles have the same color $t$. Then,
    \begin{equation*}
        \mathbb{E}[\Y] = \mathbb{E}[Y_{i_1,\ldots,j_p}^{t,\ldots,t}]= 3^r \sigma_t^{4p},
    \end{equation*}
    where the coefficient $3^r$ comes from the fact that there are $r$ repeated indices and therefore $r$ i.i.d. $N(0,1)$ random variables that are to the fourth power in $Y_{i_1,\ldots,j_p}^{t,\ldots,t}$.

    \item {\bf Case $\ell = 2$.} Without loss of generality, assume the two colors are $t_1$ and $t_2$. First of all, the maximum number of fourth powers of random variables in the product \eqref{eq:Y-prod} is $r-2$, which gives a coefficient of at most $3^{r-2}$ in front of the singular values. Due to \Cref{lem:obs}, there are only two possible cases \ref{item:case1} and \ref{item:case2}:
    \begin{enumerate}[(1)]
        \item For each $\ell = 1,\ldots,r$, the two corresponding arcs $a_{\ell}^{(i)}$ and $a_{\ell}^{(j)}$ have the same color. Moreover, not all $r$ pairs have the same color. There are $2^r-1$ color assignments of colors $t_1$ and $t_2$ that satisfy this constraint; each of them has a certain color \emph{pattern}, denoted by ``patt'', which uniquely determines the total length of arcs colored by $t_1$:
        \begin{equation} \label{eq:d-patt}
            d(\text{patt}) = \sum_{\ell \text{ such that } \atop a_{\ell}^{(i)} \text{has color }t_1} d_{\ell}^{(i)} + \sum_{\ell \text{ such that } \atop a_{\ell}^{(j)} \text{has color }t_1} d_{\ell}^{(j)}.
        \end{equation}
        Note that $2 \le d(\text{patt}) \le 2p-2$ since not all pairs of arcs have same color.
        Then, the product of \emph{singular values} in \eqref{eq:Y-prod} is
        \begin{equation} \label{eq:pattern}
            \sigma_{t_1}^{2d(\text{patt})} \sigma_{t_2}^{4p-2d(\text{patt})}.
        \end{equation}
        \item For each $\ell = 1,\ldots,r$, the two corresponding arcs $a_{\ell}^{(i)}$ and $a_{\ell}^{(j)}$ have different colors. There are $2^r$ color assignments of colors $t_1$ and $t_2$ that satisfy this constraint. 
        Each color pattern again uniquely determines a number $d(\text{patt})$ in \eqref{eq:d-patt}, and hence the the product of \emph{singular values} in \eqref{eq:Y-prod} is, again, \eqref{eq:pattern}.
    \end{enumerate}
    Then, we need to sum over all possible choices of $t_1$ and $t_2$, each of which can be anything in $\{1,\ldots, n\}$, so that
    \begin{multline*}
    \sum_{t_1,\ldots,s_p \text{ such that there are} \atop 2 \text{ distinct singular values in product \eqref{eq:Y-prod}}} \mathbb{E}[\Y]  \le 3^{r-2} \sum_{\text{valid color}\atop\text{pattern patt}}\sum_{t_1 \neq t_2} \sigma_{t_1}^{2d(\text{patt})} \sigma_{t_2}^{4p-2d(\text{patt})} \\
    \le 3^{r-2} \sum_{\text{valid color}\atop\text{pattern patt}} \|A\|_{2d(\text{patt})}^{2d(\text{patt})} \|A\|_{4p-2d(\text{patt})}^{4p-2d(\text{patt})} \le 3^{r-2} (2^{r+1}-1) \|A\|_4^4 \|A\|_{4p-4}^{4p-4}.  
    \end{multline*}
    
    \item {\bf Case $3 \leq \ell \leq r$.} 
    Denote the distinct colors by $t_1, \ldots, t_\ell$. 
    First of all, by \Cref{lem:obs}, only case \ref{item:case1} can happen: this means that, for each $\ell = 1,\ldots,r$, the two corresponding arcs $a_{\ell}^{(i)}$ and $a_{\ell}^{(j)}$ have the same color.

    We claim that the maximum number of fourth power of standard Gaussian random variables in product \eqref{eq:Y-prod} is $r - \ell$. To see this, notice that we get a fourth power of the random variable $\omega_{\bullet, i^{\ast}}$ where $i^{\ast}$ is a repeated index (circled point) if and only if the two contiguous arcs joined by $i^{\ast}$ have the same color in both circles.
    For $r$ arcs in the circle with $\ell$ distinct colors, there are at most $r-\ell+1$ adjacent arcs that can have the same color. Hence, there are at most $r - \ell$ fourth power of standard Gaussian. 
    This contributes to a coefficient at most $3^{r-\ell}$ in front of singular values in $\E[\Y]$.
    
    Since we are under \Cref{lem:obs}, case \ref{item:case1}, the color of $a_{\ell}^{(i)}$ uniquely determines the color of $a_{\ell}^{(j)}$. It therefore suffices to determine the color pattern for $i$-circle. 
    A valid color pattern needs to have all $\ell$ different colors in $i$-cycle. The total number of valid color patterns is bounded by $\ell^r$ (which is the number of ways to assign $\ell$ colors to $r$ arcs in $i$-circle).
    
    Now, given a valid color pattern ``patt'' of colors $t_1, \ldots, t_\ell$ on $2r$ arcs in the two circles, it uniquely determines the total length of arcs colored by each $t_q$, $1 \leq q \leq \ell$,
    \begin{equation*}
        d_q(\text{patt}) = \sum_{\ell \text{ such that } \atop a_{\ell}^{(i)} \text{has color }t_q} d_{\ell}^{(i)} + \sum_{\ell \text{ such that } \atop a_{\ell}^{(j)} \text{has color }t_q} d_{\ell}^{(j)}.
    \end{equation*}
    They satisfy the relations $d_q(\text{patt}) \geq 2$ and $d_1(\text{patt}) + \cdots +  d_\ell(\text{patt}) = 2p$.
    With this color pattern, the expectation of product \eqref{eq:Y-prod} is bounded above by $3^{r-\ell} \sigma_{t_1}^{2d_1(\text{patt})} \sigma_{t_2}^{2d_2(\text{patt})}  \cdots \sigma_{t_{\ell}}^{2d_\ell(\text{patt})}$.

    Finally, summing over all possible color patterns and all possible distinct colors, we have
    \begin{align*}
        &\sum_{t_1,\ldots,s_p \text{ such that there are} \atop \ell \text{ distinct singular values in product \eqref{eq:Y-prod}}} \mathbb{E}[\Y] \\
        & \le 3^{r-\ell} \sum_{\text{valid color}\atop\text{pattern patt}} \sum_{\text{ distinct} \atop t_1, \ldots, t_{\ell}} \sigma_{t_1}^{2d_1(\text{patt})} \sigma_{t_2}^{2d_2(\text{patt})}  \cdots \sigma_{t_{\ell}}^{2d_\ell(\text{patt})} \\
        & \le 3^{r-\ell} \sum_{\text{valid color}\atop\text{pattern patt}} \prod_{q=1}^{\ell} \left( \sum_t \sigma_t^{2d_q(\text{patt})} \right) 
        = 3^{r-\ell} \sum_{\text{valid color}\atop\text{pattern patt}} \prod_{q=1}^{\ell} \|A\|_{2d_q(\text{patt})}^{2d_q(\text{patt})} \\
        & \le 3^{r-\ell} \sum_{\text{valid color}\atop\text{pattern patt}} \|A\|_4^{4(\ell-1)} \|A\|_{4p-4(\ell-1)}^{4p-4(\ell-1)} \le 3^{r-\ell} \ell^r \|A\|_{4}^{4(\ell-1)} \|A\|_{4p-4(\ell-1)}^{4p-4(\ell-1)}.
    \end{align*}
    where the second to the last inequality is obtained by repeatedly applying \Cref{lemma:inequalities}.
\end{itemize}    
Summing all the contributions of $\E[\Y]$ from the above cases, we have
\begin{equation*}
    \mathbb{E}[f_{i_1,\ldots,j_p}] \le \sum_{\ell=1}^r c(r, \ell)
\end{equation*}
for the function $c(r, \ell)$ defined in the statement of the lemma.
\end{proof}

\subsubsection{Second-order analysis}\label{sec:secondorder}
Now, we are ready to prove the bounds in \Cref{thm:secondorderbound} and \Cref{cor:secondorder}.

\begin{proof}[Proof of \Cref{thm:secondorderbound}]
Recalling~\eqref{eq:splitXsquare} and the fact that $\hat \Theta_{2p}$ is an unbiased estimator for $\binom{k}{p} \| A \|_{2p}^{2p}$ and \eqref{eq:splitXsquare}, we have
\begin{equation*}
\var(\hat \Theta_{2p}) = \E[\hat \Theta_{2p}^2] - \left(\E[\hat \Theta_{2p}] \right)^2 = F_0 + F_1 + F_2 + \sum_{r=3}^p F_r - \binom{k}{p}^2 \| A \|_{2p}^{4p}.
\end{equation*}
The desired bound in the statement of theorem immediately follows by applying \Cref{lemma:rrepeated,lemma:r0,lemma:r1,lemma:r2,lemma:rlarger} to obtain
\begin{align*}
\E[\hat \Theta_{2p}^2] &= F_0 + F_1 + F_2 + \sum_{r=3}^p F_r \\
& \le \binom{k}{2p} \binom{2p}{p} \|A\|_{2p}^{4p}
+ \binom{k}{2p-1}\binom{2p-1}{p} p \left ( \|A\|_{2p}^{4p} + 2 \|A\|_{4p}^{4p} \right ) \\
&\quad + \binom{k}{2p-2}\binom{2p-2}{p}\binom{p}{2} \left ( 6 \|A\|_{4p}^{4p} + 3 \|A\|_4^4\|A\|_{4p-4}^{4p-4}\right ) \\
&\quad + \sum_{r=3}^p \binom{k}{2p-r}\binom{2p-r}{p}\binom{p}{r} \sum_{\ell = 1}^r c(r,\ell) \|A\|_4^{4(\ell-1)} \|A\|_{4p-4(\ell-1)}^{4p-4(\ell-1)}. \qedhere
\end{align*}
\end{proof}

\begin{proof}[Proof of \Cref{cor:secondorder}]
Note that $\var(\KV) = \binom{k}{p}^{-2} \var(\hat \Theta_{2p})$.
By \Cref{thm:secondorderbound} and the fact $\binom{k}{2p-r}\binom{2p-r}{p}\binom{p}{r} = \calO(k^{2p-r})$, we have
{\small
\begin{align}
\var(\KV) & \le - \|A\|_{2p}^{4p} + \binom{k}{p}^{-2} \binom{k}{2p} \binom{2p}{p} \|A\|_{2p}^{4p} \nonumber \\
&\quad + \binom{k}{p}^{-2} \binom{k}{2p-1}\binom{2p-1}{p} p \left ( \|A\|_{2p}^{4p} + 2 \|A\|_{4p}^{4p} \right ) \nonumber \\
&\quad + \binom{k}{p}^{-2} \binom{k}{2p-2}\binom{2p-2}{p}\binom{p}{2} \left ( 6 \|A\|_{4p}^{4p} + 3 \|A\|_4^4\|A\|_{4p-4}^{4p-4}\right ) + \calO \left( \frac{1}{k^3} \right). \label{eq:pf-KV-bound}
\end{align} 
}
For the second term in \eqref{eq:pf-KV-bound}, performing a second-order expansion in $1/k$, we have
\begin{align}
\binom{k}{p}^{-2} \binom{k}{2p} \binom{2p}{p} \|A\|_{2p}^{4p}
& = \frac{(k-p)(k-p-1)\cdots(k-2p+1)}{k(k-1)\cdots (k-p+1)} \|A\|_{2p}^{4p} \nonumber \\
& = \|A\|_{2p}^{4p}\prod_{i=0}^{p-1} \left ( 1 - \frac{p}{k-i} \right )  \nonumber \\
& = \left ( 1 - \frac{p^2}{k} + \frac{p^2(p-1)(p-2)}{2k^2} \right ) \|A\|_{2p}^{4p} + \mathcal O\left ( \frac{1}{k^3} \right ). \label{eq:F0-bound}
\end{align}

For the third term in \eqref{eq:pf-KV-bound}, using a first-order expansion of $\prod_{i=1}^{p-1} \left ( 1 - \frac{p-1}{k-i} \right )$, we have
\begin{align}
&\quad \binom{k}{p}^{-2} \binom{k}{2p-1}\binom{2p-1}{p} p \left ( \|A\|_{2p}^{4p} + 2 \|A\|_{4p}^{4p} \right ) \nonumber \\
&= \frac{p^2(k-p)(k-p-1) \cdots (k-2p+2)}{k(k-1)\cdots (k-p+1)} \left ( \|A\|_{2p}^{4p} + 2 \|A\|_{4p}^{4p} \right ) \nonumber \\
&= \frac{p^2}{k}    \left ( \|A\|_{2p}^{4p} + 2 \|A\|_{4p}^{4p} \right ) \prod_{i=1}^{p-1} \left ( 1 - \frac{p-1}{k-i} \right ) \nonumber \\
& =  \left (\frac{p^2}{k} - \frac{p^2(p-1)^2}{k^2} \right )   \left ( \|A\|_{2p}^{4p} + 2 \|A\|_{4p}^{4p} \right )   + \mathcal O \left ( \frac{1}{k^3} \right ).\label{eq:F1-bound}
\end{align}

For the fourth term in \eqref{eq:pf-KV-bound}, we have
\begin{align}
&\quad \binom{k}{p}^{-2} \binom{k}{2p-2}\binom{2p-2}{p}\binom{p}{2} \left ( 6 \|A\|_{4p}^{4p} + 3 \|A\|_4^4\|A\|_{4p-4}^{4p-4}\right ) \nonumber \\
&= \frac{p^2 (p-1)^2}{2k(k-1)} \frac{(k-p)(k-p-1) \cdots (k-2p+3)}{(k-2)(k-3) \cdots (k-p+1)} \left ( 6 \|A\|_{4p}^{4p} + 3 \|A\|_4^4\|A\|_{4p-4}^{4p-4}\right ) \nonumber \\
&= \frac{p^2 (p-1)^2}{2k^2} \left ( 6 \|A\|_{4p}^{4p} + 3 \|A\|_4^4\|A\|_{4p-4}^{4p-4}\right ) + \calO\left(\frac{1}{k^3}\right). \label{eq:F2-bound}
\end{align}

Finally, putting together \eqref{eq:pf-KV-bound}--\eqref{eq:F2-bound}, we have
\begin{align*}
\mathrm{Var}(\KV)
&\le - \|A\|_{2p}^{4p} + \left ( 1 - \frac{p^2}{k} + \frac{p^2(p-1)(p-2)}{2k^2}\right ) \|A\|_{2p}^{4p} \\
&\quad +  \left ( \frac{p^2}{k} - \frac{p^2(p-1)^2}{k^2} \right ) \left ( \|A\|_{2p}^{4p} + 2 \|A\|_{4p}^{4p} \right ) \\
&\quad + \frac{p^2 (p-1)^2}{2k^2} \left ( 6 \|A\|_{4p}^{4p} + 3 \|A\|_4^4\|A\|_{4p-4}^{4p-4}\right ) + \calO\left(\frac{1}{k^3}\right) \\
&\le  \frac{2p^2}{k} \|A\|_{4p}^{4p} 
+ \frac{p^2 (p-1)^2}{k^2} \left (\|A\|_{4p}^{4p} + \frac{3}{2} \|A\|_4^4\|A\|_{4p-4}^{4p-4} - \frac{1}{2}\|A\|_{2p}^{4p} \right ). \qedhere
\end{align*}
\end{proof}

\section{Numerical examples}\label{sec:examples}

We numerically illustrate the performance of our bounds and estimates on the variance of $\KV$. The numerical experiments have been performed in \rev{Python}, and all the figures in this document can be reproduced using the code available at \url{https://github.com/Alice94/schatten-norm}. Recall that, without loss of generality, we can work with diagonal matrices. In the plots below, we estimate the exact variance of $\KV$ for a fixed matrix and a fixed value of $k$ by running $\KV$ for \rev{$100000$ times if $k < 100$ and for $3000$ times if $k \ge 100$, }and computing the sample variance of the results. \rev{The sample variance} corresponds to the blue line in the figures in this section. The \rev{orange} lines (first-order estimate) correspond to Theorem~\ref{thm:firstorder}, \rev{ that is, we approximate the variance as $2p^2 \|A\|_{4p}^{4p}/k$; the green lines (second-order estimate) correspond to Theorem~\ref{cor:secondorder}, that is, we approximate the variance as
\begin{equation*}
    \frac{2p^2\|A\|_{4p}^{4p}}{k} + \frac{p^2(p-1)^2}{k^2}\left ( \|A\|_{4p}^{4p} + \frac{3}{2} \|A\|_4^4 \|A\|_{4p-4}^{4p-4} - \frac{1}{2}\|A\|_{2p}^{4p} \right );
\end{equation*}
the red dotted lines correspond to the bound in Theorem~\ref{thm:secondorderbound}, and the purple dotted lines correspond to the bound~\eqref{eq:varKV-original}.} In all our examples, we keep the matrix size quite small ($100 \times 100$) so that we are able to run the estimator many times to compute its empirical variance as precisely as possible.

\begin{example}\label{ex:0.8}
    \rev{As a first example, we consider a matrix with singular values that exhibit exponential decay. We take $A$ to be the diagonal matrix $A \in \mathbb{R}^{100 \times 100}$} with diagonal entries $0.8, 0.8^2, 0.8^3, \ldots, 0.8^{100}$. \Cref{fig:0.8} illustrates the variance of $\KV$ for $p = 4, 6, 8$ with values of $k$ ranging from \rev{$10$ to $1280$}. Note that the first-order estimate and the second-order estimates tend to underestimate the variance and get better for larger values of $k$, as expected.    
\begin{figure}
\centering
\includegraphics[scale=.2]{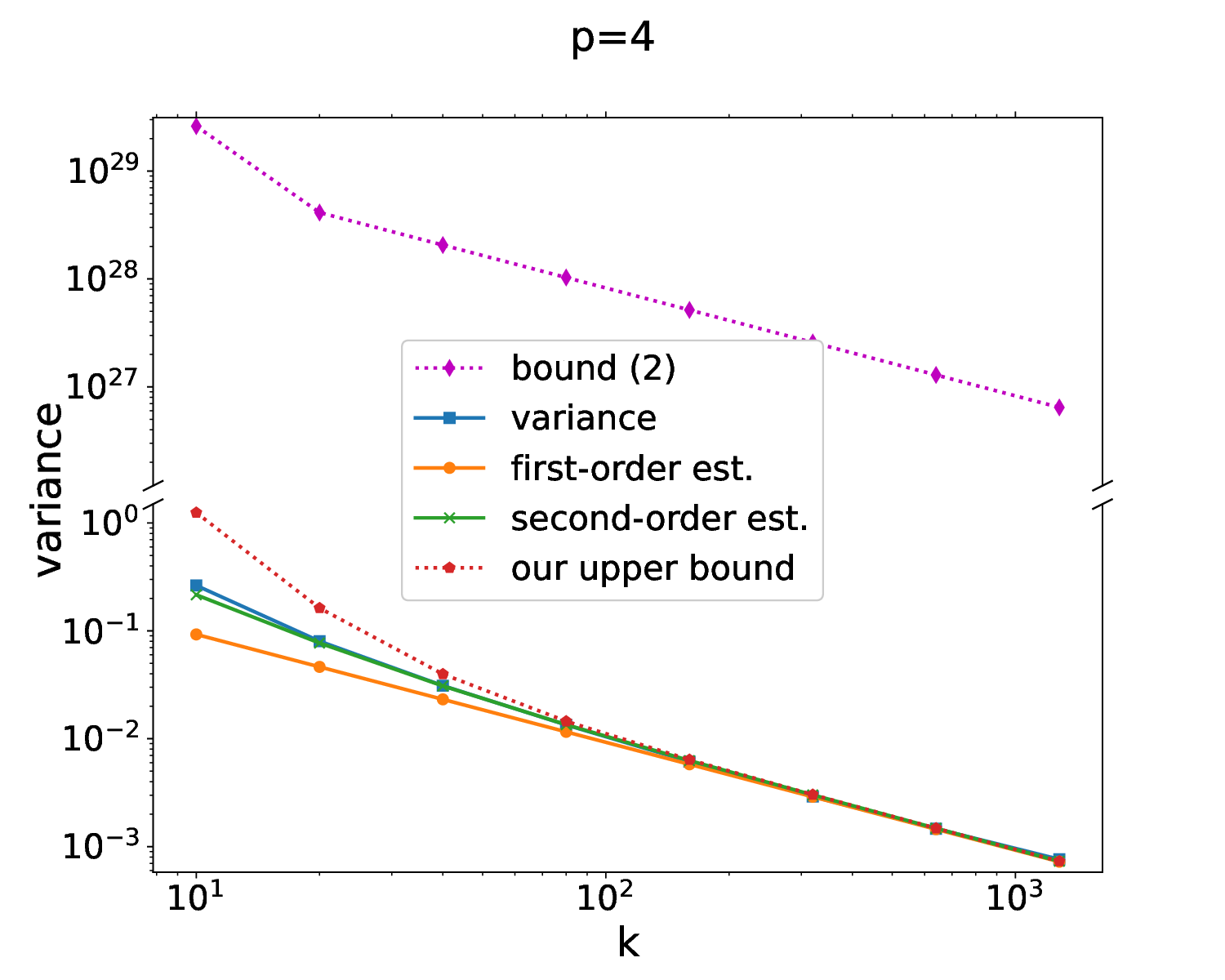}\includegraphics[scale=.2]{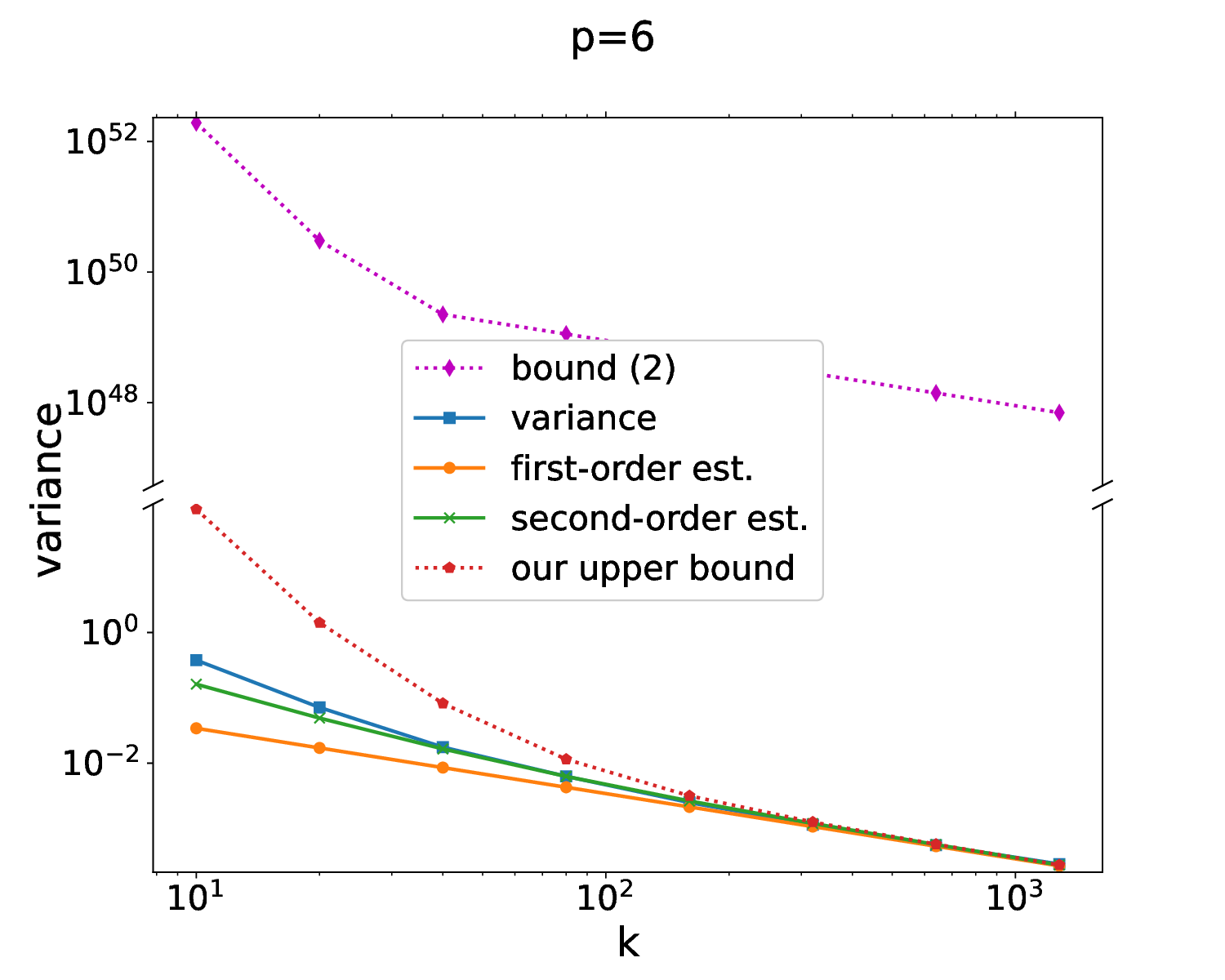}\includegraphics[scale=.2]{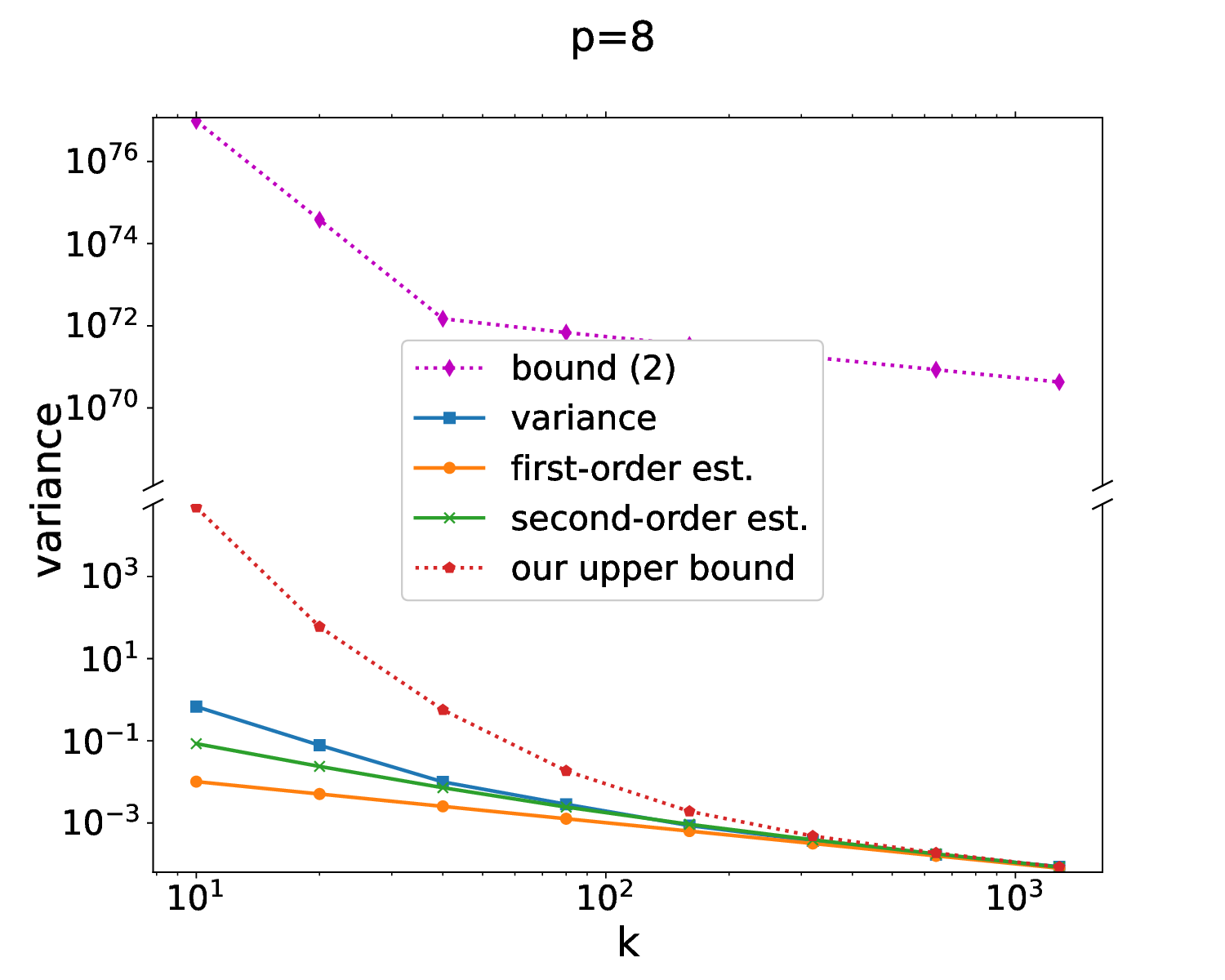}
\caption{Comparison of the variance of $\KV$, our bounds, our estimates, \rev{and~\eqref{eq:varKV-original}} for the matrix from Example~\ref{ex:0.8}.}
\label{fig:0.8}
\end{figure}
\end{example}

\begin{example}\label{ex:1overi}
\rev{We now illustrate the behavior of our bounds for matrices whose singular values exhibit algebraic decay. In Figure~\ref{fig:1overi} we consider the $100 \times 100$ diagonal matrix with 
diagonal entries $1, 1/4, 1/9, 1/16, \ldots, 1/10000$, for $p = 3, 5, 7$, and in Figure~\ref{fig:alg4} we consider the $100 \times 100$ diagonal matrix with diagonal entries $1, 1/2^4, 1/3^4, \ldots, 1/100^4$ for $p=2,6,10$.} 

\begin{figure}[htb]
\includegraphics[scale=.2]{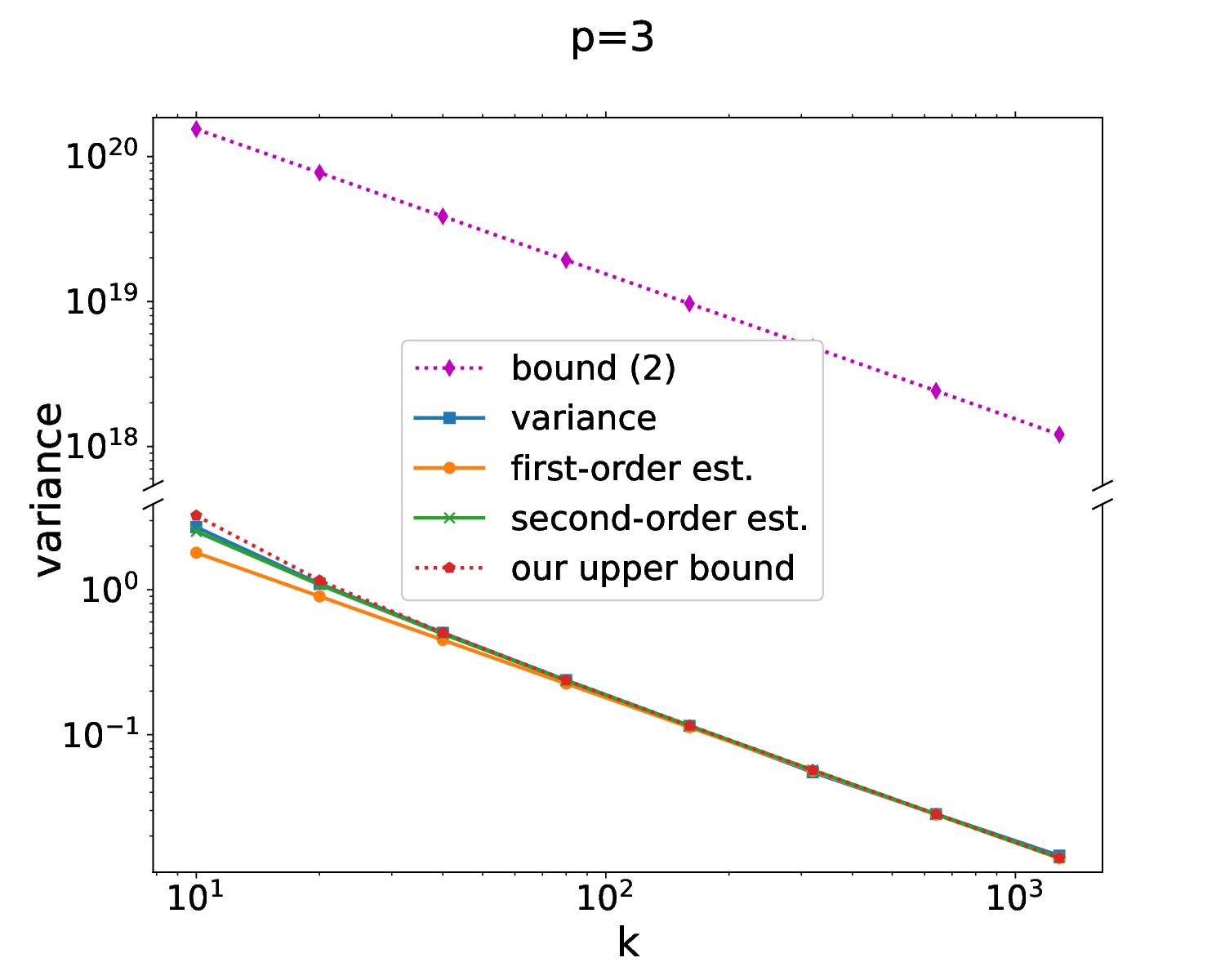}\includegraphics[scale=.2]{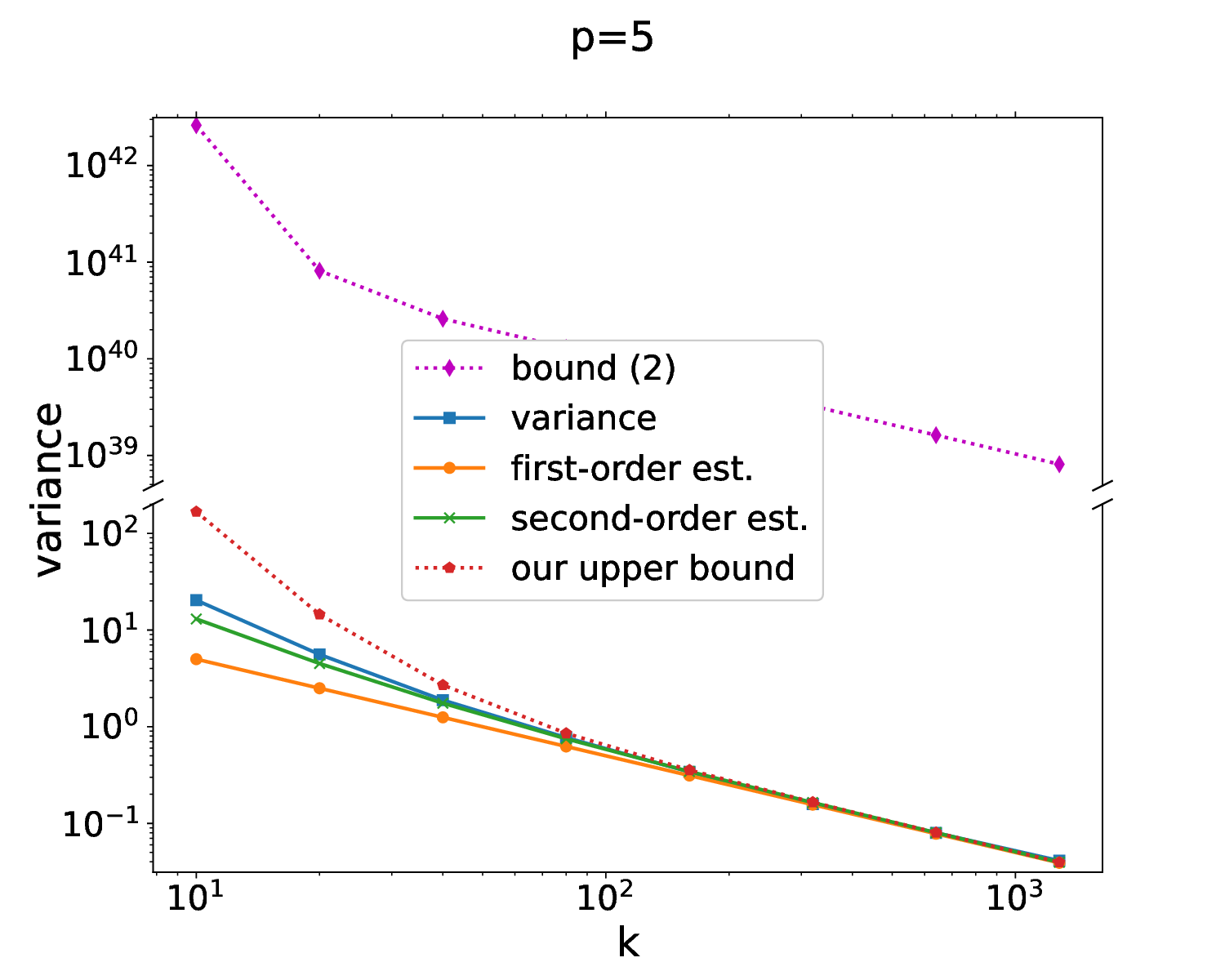}\includegraphics[scale=.2]{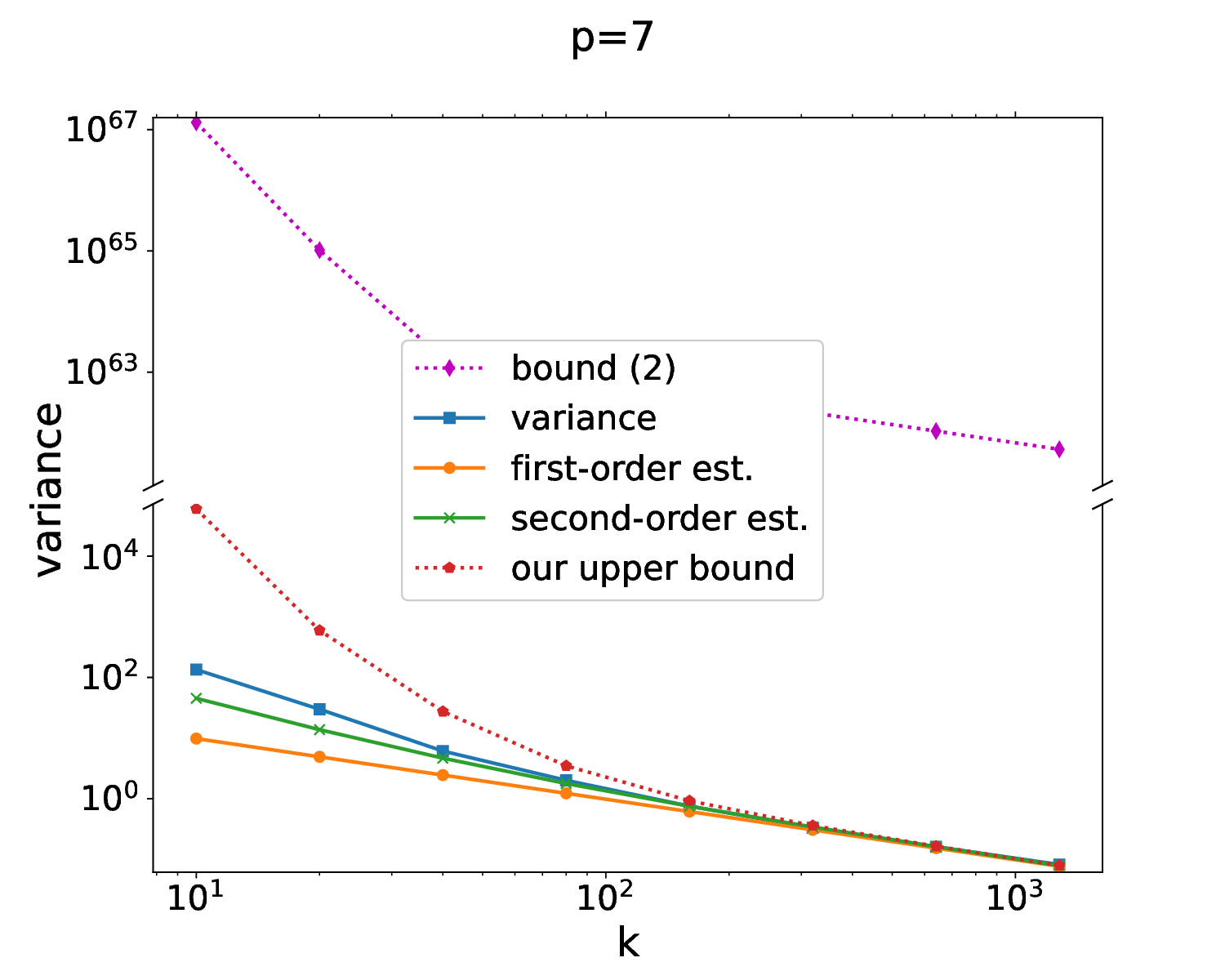}
\caption{Comparison of the normalized variance of $\KV$, our bounds, our estimates, \rev{and~\eqref{eq:varKV-original}} for the matrix from Example~\ref{ex:1overi} with diagonal entries $1, 1/4, 1/9, \ldots, 1/10000$.}
\label{fig:1overi}
\end{figure}

\begin{figure}[htb]
\includegraphics[scale=.2]{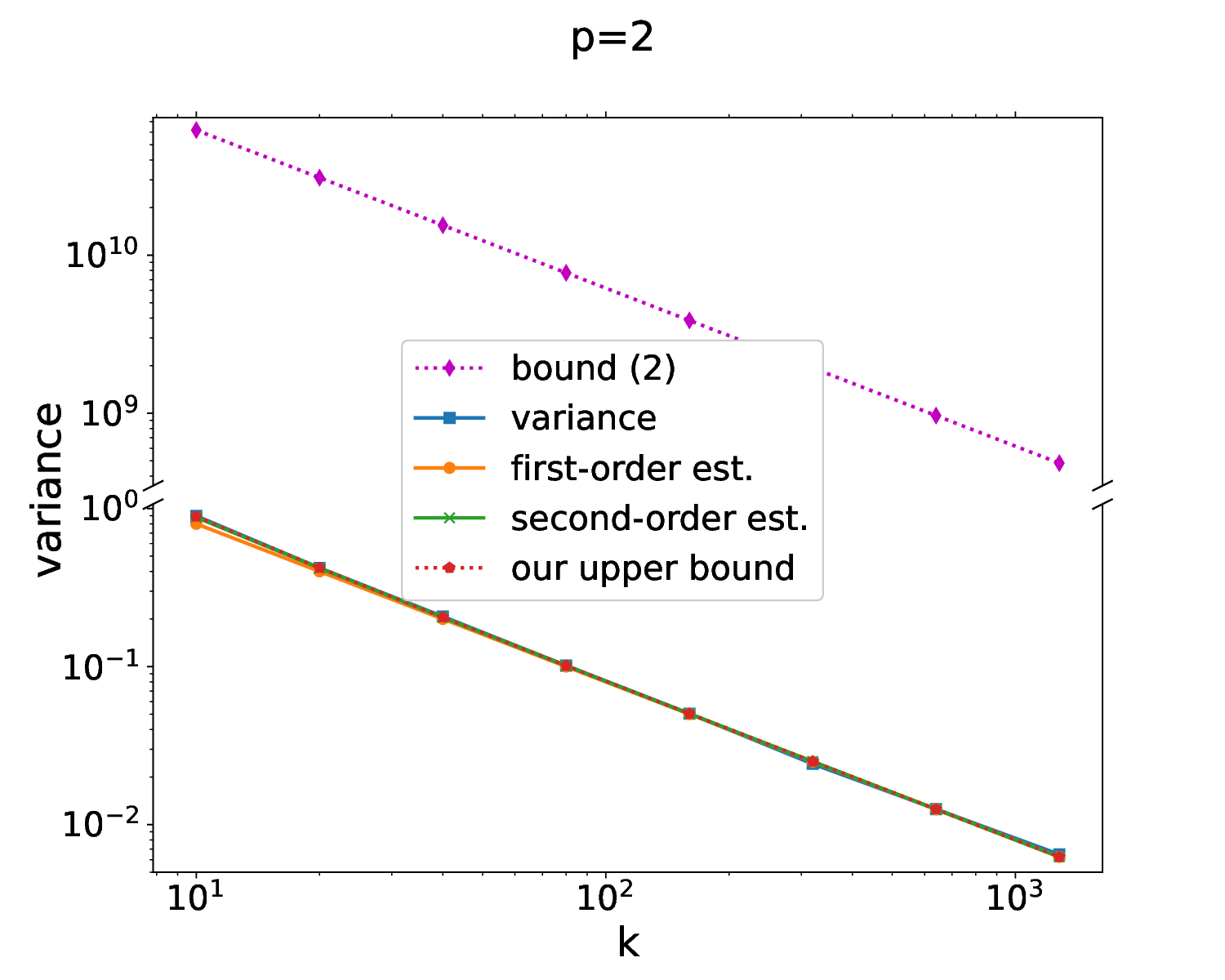}\includegraphics[scale=.2]{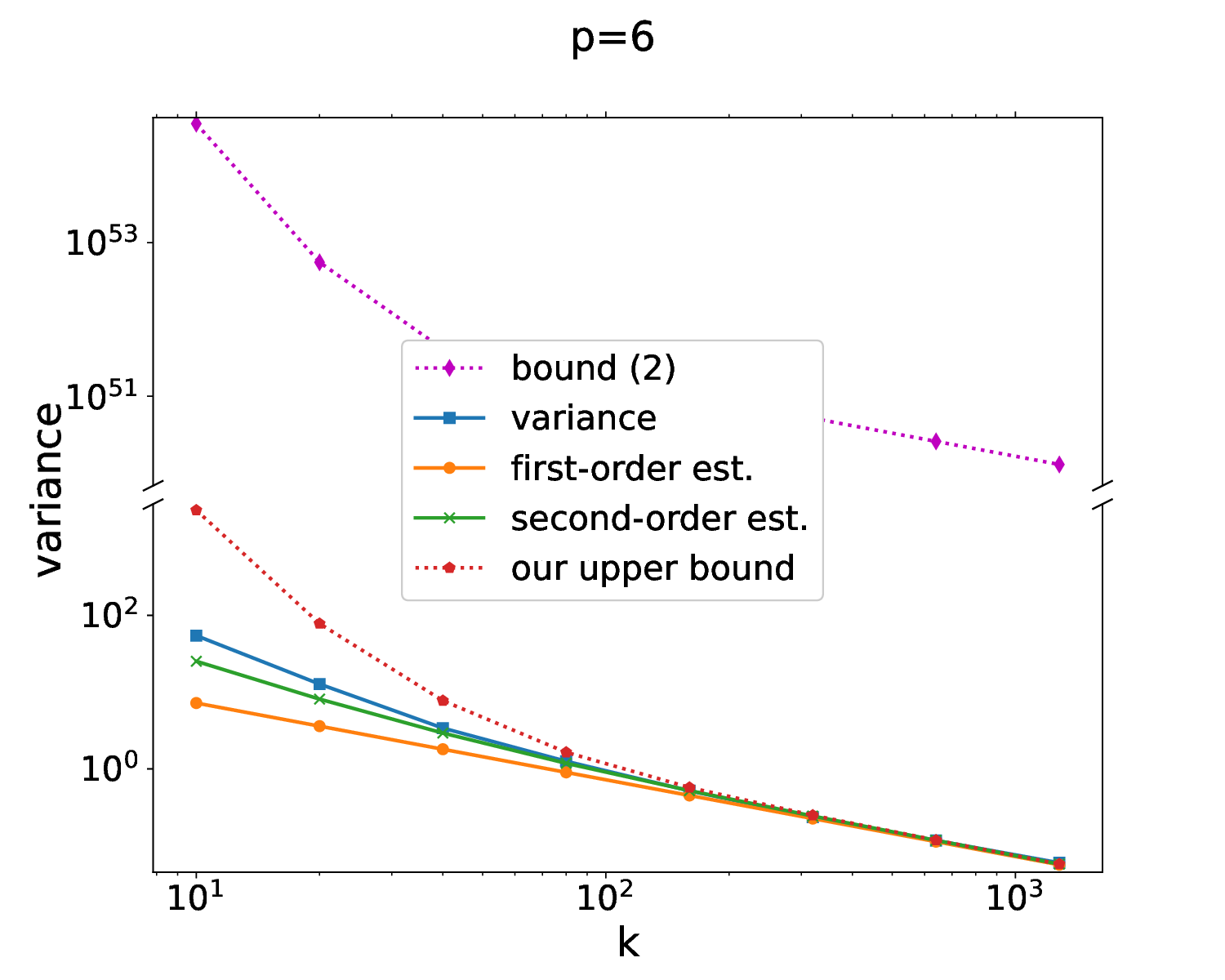}\includegraphics[scale=.2]{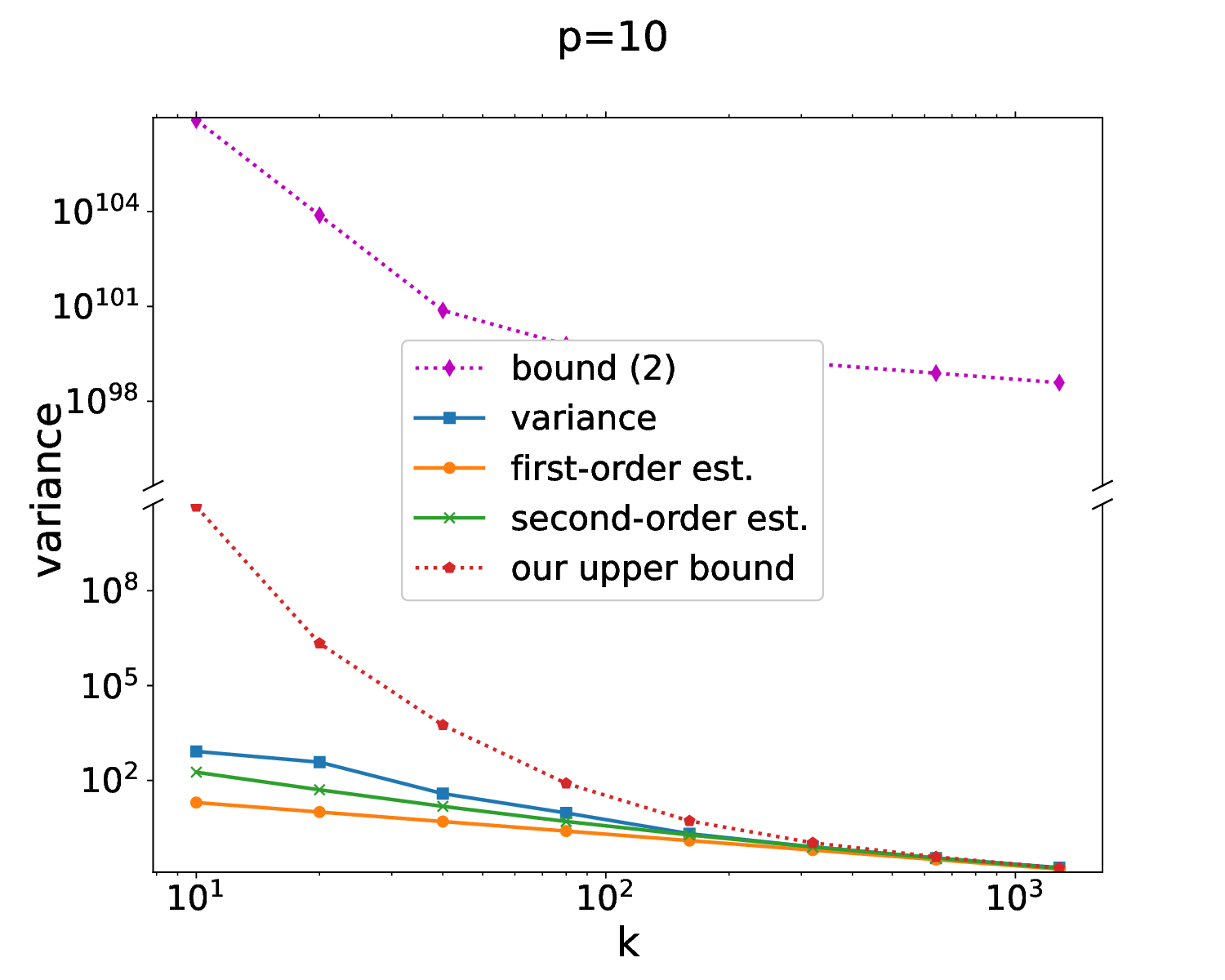}
\caption{Comparison of the variance of $\KV$, our bounds, our estimates, \rev{and~\eqref{eq:varKV-original}} for the matrix from Example~\ref{ex:1overi} with diagonal entries $1, 1/2^4, 1/3^4, \ldots, 1/100^4$.}
\label{fig:alg4}
\end{figure}
\end{example}

\begin{example}\label{ex:identity}
    The motivation for this work is to explore the behavior of $\KV$ for numerically \rev{low-}rank matrices. However, even for the matrices with no eigenvalue decay, our bound from \Cref{thm:secondorderbound} improves the existing bound \eqref{eq:varKV-original} from \cite{kong2017spectrum}. In \Cref{fig:identity}, we consider $A = I_{100}$, the identity matrix, and compare the empirical variance of $\KV$ with our estimates, our bounds, and the bound \eqref{eq:varKV-original}. Our bounds are not tight at all, especially for larger values of $p$, but are several orders of magnitude tighter than the bound \eqref{eq:varKV-original}.

\begin{figure}[htb]
\centering
\includegraphics[scale=.2]{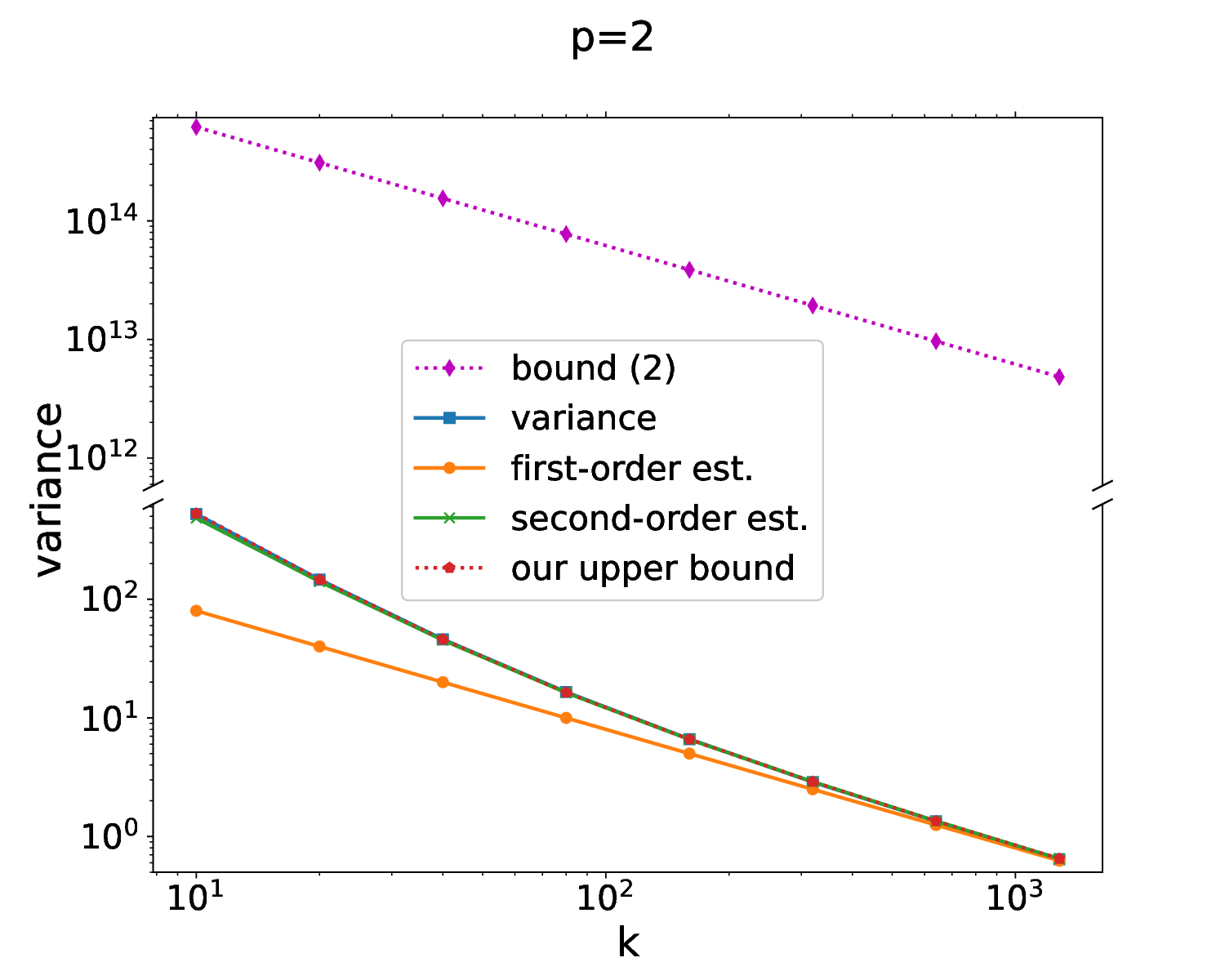}\includegraphics[scale=.2]{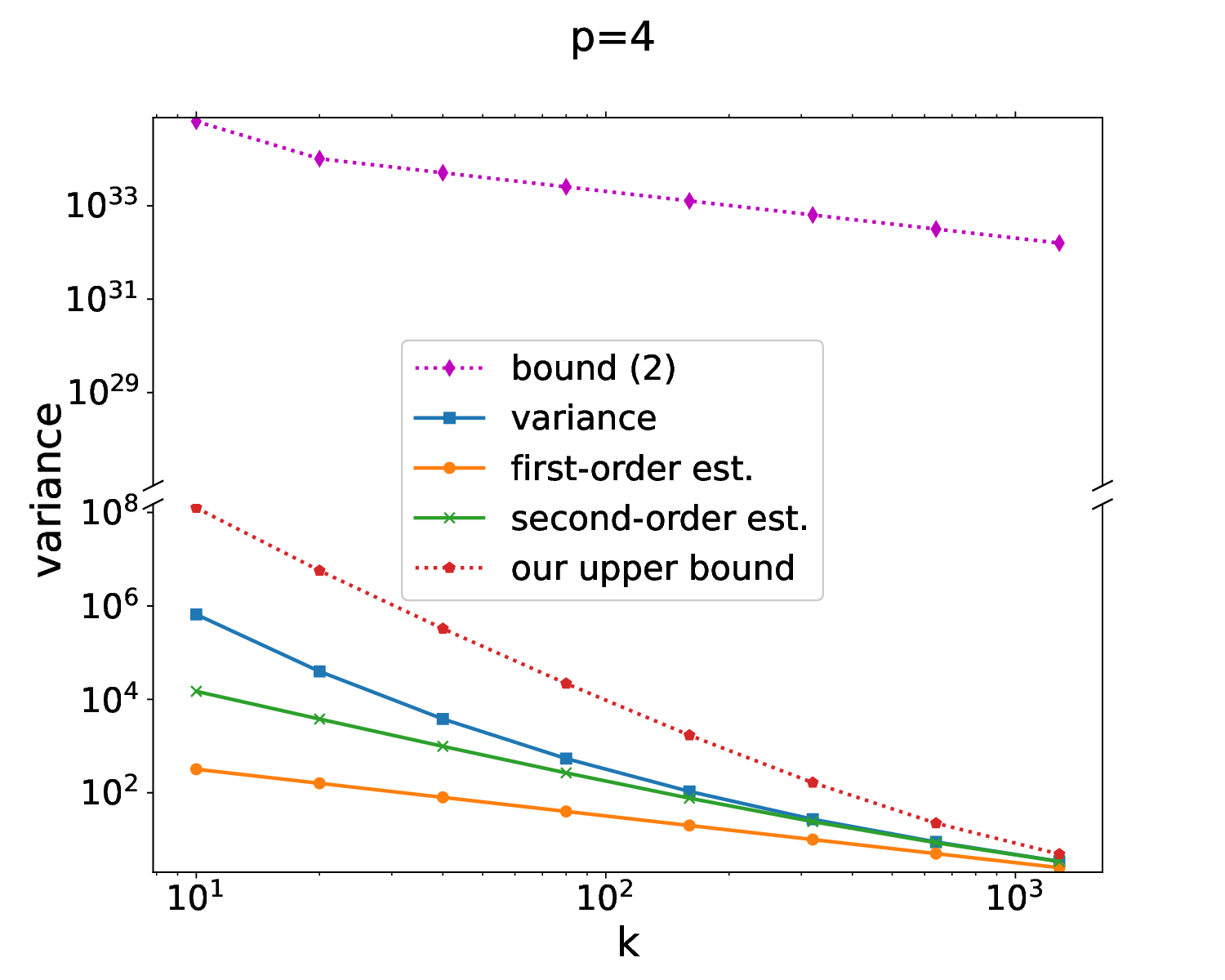}\includegraphics[scale=.2]{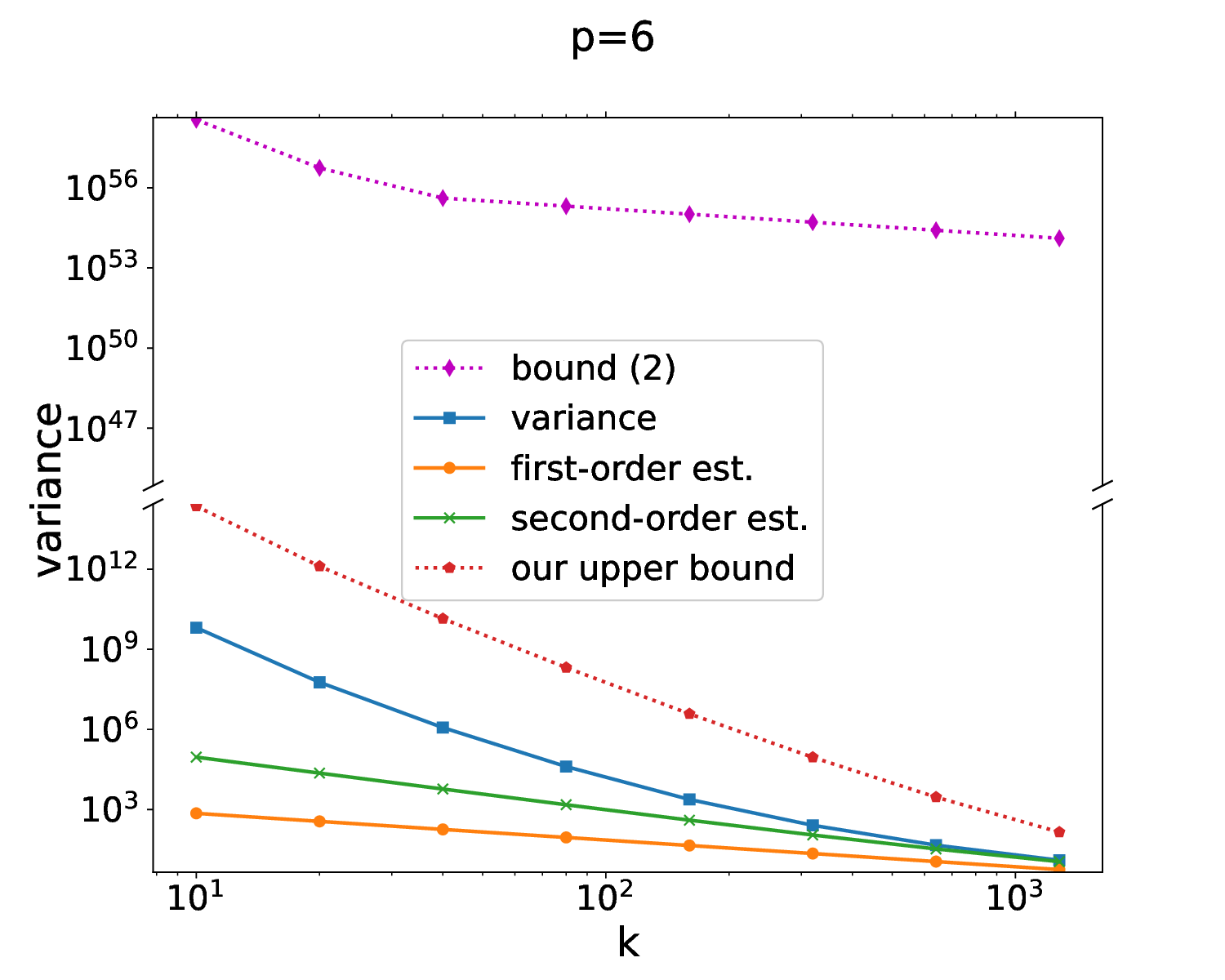}
\caption{Comparison of the variance of $\KV$, our bounds, our estimates, and \eqref{eq:varKV-original} for the identity matrix.}
\label{fig:identity}
\end{figure}
\end{example}

\begin{example}\label{ex:varyingn}
    \rev{We now consider a sequence of matrices of increasing size and same spectral density, and investigate the behavior of our bound and estimates. More precisely, we consider the diagonal matrix $A_1 \in \mathbb{R}^{10 \times 10}$ with diagonal elements $1, 1/2^4, \ldots, 1/10^4$, and the matrices $A_t = \begin{bmatrix} A_{t-1} & 0 \\ 0 & A_{t-1}\end{bmatrix}$ of increasing dimension, for $t=2,\ldots,10$, so that $A_t$ has size $10\cdot 2^{t-1} \times 10\cdot 2^{t-1}$. We fix $p = 5$ and, for each matrix, we plot in Figure~\ref{fig:varyingn} the variance, the first-order and second-order estimates, our upper bound, and the bound~\eqref{eq:varKV-original}, all of them divided by $\|A\|_{10}^{20}$. The number of samples is $k=100$. Note that the numerical rank increases linearly with the size of the matrix; this explains why our bounds and estimates are better for the matrices of smaller size.}

\begin{figure}[htb]
\centering
\includegraphics[scale=.3]{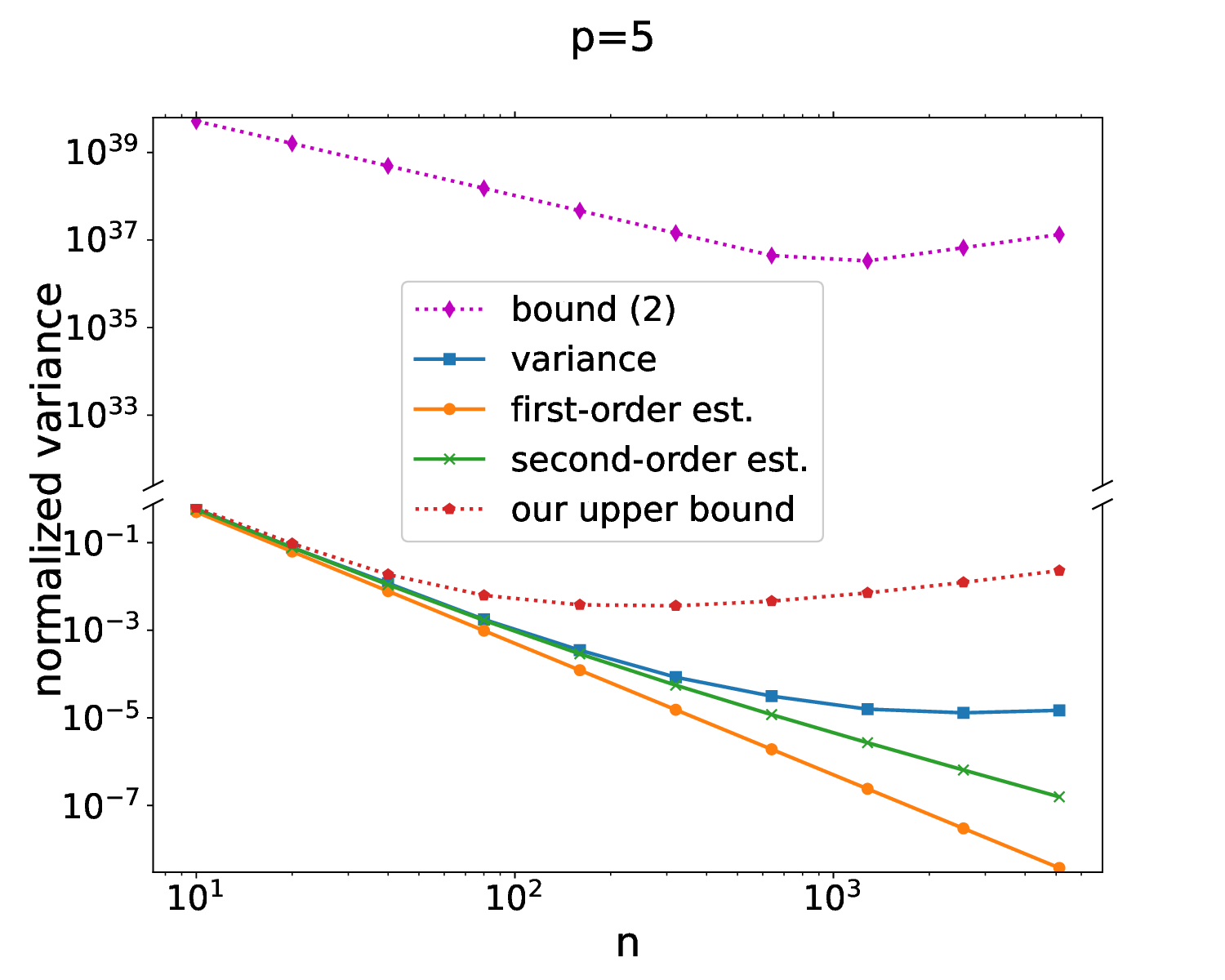}
\caption{Comparison of the variance of $\KV$, our bounds, our estimates, and \eqref{eq:varKV-original} for the matrices from Example \ref{ex:varyingn}.}
\label{fig:varyingn}
\end{figure}
\end{example}

\section{Discussion}\label{sec:conclusion}

In this work, we analyzed the variance of $\KV$ given by \Cref{alg:KV}, which is an unbiased estimator for the $2p$-th power of the Schatten-$2p$ norm of a matrix $A$, assuming only a sketch $Y = A \Omega$ with a Gaussian matrix $\Omega$ is available. We provided a new bound -- \Cref{thm:secondorderbound} -- that improves the bound \eqref{eq:varKV-original} from \cite{kong2017spectrum} for any matrix. Moreover, we focused on the matrices that have strong singular value decay, and proposed a first-order and second-order estimate on the variance. Numerically, we observe that these estimates are rather precise for moderate values of $p$ and $k$. For matrices with no singular value decay, the bound holds but the estimates do not really describe the behavior of $\KV$ since they are too optimistic. \rev{In all cases, our upper bound on the variance was far tighter than the bound~\eqref{eq:varKV-original}. After our paper was submitted, a preprint addressing the same problem of computing or approximating the variance of $\KV$ was posted to arXiv~\cite{Horesh2024}; the upper bound in~\cite{Horesh2024}, however, explicitly depends on the size of the matrix, and is not tailored to the case of numerically low-rank matrices.}

The \rev{sketching model} we considered is applicable in scenarios where we need to extract information on the moments of a covariance matrix of a multivariate Gaussian distribution \rev{for which only some samples drawn from the distribution are available.} This is the original setting considered in \cite{kong2017spectrum}. \Cref{alg:KV} was also suggested as a cheap Schatten-$2p$ norm estimator in the randomized numerical linear algebra review papers \cite{Martinsson2020,Murray2023}. We remark that, if the matrix $A$ can be accessed more than once via matrix-vector product, one should use other techniques to estimate Schatten norms; for instance, the ones based on Hutchinson trace estimator since they have smaller variance.

\paragraph{Acknowledgments} \rev{The authors thank the anonymous referees for providing helpful comments that improved the quality of this work. }

\bibliographystyle{abbrv}
\bibliography{bib}

\end{document}